\documentclass[12pt]{article}
\usepackage{amsfonts,amssymb,latexsym,amsmath,amsthm}
\usepackage{eucal}
\usepackage{geometry}
\usepackage{hyperref}
\usepackage{enumitem}   
\usepackage{xcolor}
\usepackage{soul}
\usepackage{upgreek}
\usepackage[colorinlistoftodos]{todonotes}
\usepackage{authblk}

\newtheorem{thm}{Theorem}
\newtheorem*{thm*}{Theorem}
\newtheorem{theorem}{Theorem}

\newtheorem{prop}{Proposition}
\newtheorem{coro}{Corollary}
\newtheorem{lemma}{Lemma}
\newtheorem{defn}{Definition}
\newtheorem{remark}{Remark}

\def\<{\langle}
\def\>{\rangle}

\newcommand{\R}{\mathbb{R}}

\newcommand\be{\begin{equation}} 
\newcommand\ee{\end{equation}}
\newcommand{\comment}[1]

\tikzset { domaine/.style 2 args={domain=#1:#2} }

\newtheorem*{theo*}{Theorem}

\newtheorem*{prop*}{Proposition}

\newtheorem*{cor*}{Corollary}

\def\bea{\begin{eqnarray*} }

\def\eea{\end{eqnarray*} }

\begin{document}

\title{A Positive Energy Theorem for Fourth-Order Gravity}
%\affil[1]{Federal University of Cear\'a, Mathematics Department, Fortaleza, Cear\'a, Brazil.}
%\affil[2]{Institut Math\'ematique de Jussieu, Universit\'e de Paris, B\^atiment Sophie Germain, Case 7052, 75205 Paris C\'edex 13, France \& DMA, Ecole normale sup\'erieure, CNRS, PSL Research University, 75005 Paris.}

\author{R. Avalos\thanks{Federal University of Cear\'a, Mathematics Department, Fortaleza, Cear\'a, Brazil}, P. Laurain\thanks{Institut Math\'ematique de Jussieu, Universit\'e de Paris, B\^atiment Sophie Germain, Case 7052, 75205 Paris C\'edex 13, France \& DMA, Ecole normale sup\'erieure, CNRS, PSL Research University, 75005 Paris.}, J. H. Lira$^{*}$} 
{%Partially supported by CNPq and FUNCAP.}}
%\date{}
\maketitle

%\begin{abstract} 
%\end{abstract}

%
%%\usepackage{dcolumn}
%\documentclass[aps,preprint,preprintnumbers,amsmath,amssymb,nofootinbib]{revtex4}
%%\documentclass[aps,prep,nofootinbib]{revtex4}
%%%%%%%%%%%%%%%%%%%%%%%%%%%%%%%%%%%%%%%%%%%%%%%%%%%%%%%%%%%%%%%%%%%%%%%%%%%%%%%%%%%%%%%%%%%%%%%%%%%%%%%%%%%%%%%%%%%%%%%%%%%%%
%\usepackage{amssymb}
%\usepackage{amsmath}
%\usepackage{bm}
%\usepackage{amsthm}

%\theoremstyle{plain}
%\newtheorem{thm}{Theorem} % reset theorem numbering for each chapter

%\theoremstyle{definition}
%\newtheorem{defn}{Definition} % definition numbers are dependent on theorem numbers
%\newtheorem{exmp}{Example} % same for example numbers

%\theoremstyle{proposition}
%\newtheorem{prop}{Proposition} % definition numbers are dependent on theorem numbers

%\theoremstyle{lemma}
%\newtheorem{lemma}{Lemma} % definition numbers are dependent on theorem numbers

%\theoremstyle{corollary}
%\newtheorem{coro}{Corollary} % definition numbers are dependent on theorem numbers

\setcounter{MaxMatrixCols}{10}

\begin{abstract}
In this paper we prove a positive energy theorem related to fourth-order gravitational theories, which is a higher-order analogue of the classical ADM positive energy theorem of general relativity. We will also show that, in parallel to the corresponding situation in general relativity, this result intersects several important problems in geometric analysis. For instance, it underlies positive mass theorems associated to the Paneitz operator, playing a similar role in the positive $Q$-curvature conformal prescription problem as the Schoen-Yau positive energy theorem does for the Yamabe problem. Several other links to $Q$-curvature analysis and rigidity phenomena are established.
\end{abstract}

%\listoftodos

%\todo[author=Rodrigo]{After the review, write the abstract.}

\section{Introduction}

In this paper we will analyse the properties of a recently proposed energy associated to higher-order gravitational theories in the stationary limit, see \cite{ALM}. Specifically, in parallel work, we have analysed gravitational theories described on globally hyperbolic space-times $(V\doteq M\times\mathbb{R},\bar{g})$ by an action functional of the form
\begin{align}\label{action}
S(\bar{g})=\int_{V}\left(\alpha R^2_{\bar{g}} + \beta\< \mathrm{Ric}_{\bar{g}},\mathrm{Ric}_{\bar{g}} \> \right)dV_{\bar{g}},
\end{align}
where $\alpha$ and $\beta$ are free parameters in this variational setting. Let us highlight that the study of these kinds of higher-order gravitational action functional is well-motivated within contemporary theoretical physics as they appear in connection with effective field theories of gravity \cite{EffectiveFT2,EffectiveFT1,Renormalisation}, as well as in the context of inflationary cosmology \cite{Inflation} and certain approaches to quantum gravity, such as conformal gravity \cite{Maldacena,Conformalgrav}.

In order to make sense of the above functional, we can assume that the class of metrics considered above are such that $R^{2}_{\bar{g}}$ and $\< \mathrm{Ric}_{\bar{g}},\mathrm{Ric}_{\bar{g}} \>$ are integrable. Then, the functional $\bar{g}\mapsto S(\bar{g})$ is well-defined and we have an $L^2$-gradient for this functional, given by a divergence-free tensor field $A_{\bar{g}}\in \Gamma(T^0_2V)$, which is explicitly given by 
\begin{align}\label{A-tensor}
\begin{split}
A_{\bar{g}}&=\beta\Box_{\bar{g}}{\mathrm{Ric}_{\bar{g}}} + (\frac{1}{2} \beta  + 2\alpha)\Box_{\bar{g}}R\: {\bar{g}} - (2\alpha +  \beta)\bar{\nabla}^2R_{\bar{g}}  - 2\beta{\mathrm{Ric}_{\bar{g}}}_{\cdot}\mathrm{Riem}_{\bar{g}}    \\
&  + 2\alpha R_{\bar{g}}\mathrm{Ric}_{\bar{g}} -\frac{1}{2 }\alpha R^2_{\bar{g}} {\bar{g}} -\frac{1}{2 } \beta\langle\mathrm{Ric}_{\bar{g}},\mathrm{Ric}_{\bar{g}}\rangle_{\bar{g}} {\bar{g}},
\end{split}
\end{align}
where above we denoted ${{\mathrm{Ric}_{\bar{g}}}_{\cdot}\mathrm{Riem}_{\bar{g}}}_{ij}\doteq {\mathrm{Ric}_{\bar{g}}}^{kl}{\mathrm{Riem}_{\bar{g}}}{}_{kijl}$. In a parallel situation to what is well-known in the context of general relativity (GR), we have shown that there is a canonical notion of energy, which we denote by $\mathcal{E}_{\alpha,\beta}(\bar{g})$, associated to asymptotically Euclidean (AE) solutions of the space-time field equations $A_{\bar{g}}=0$ which arise as perturbations of solutions $\bar{g}_0$ which possess a time-like Killing field. Although the analysis of such an energy could be quite involved in general, we are able to identify some particular choices of the parameters $\alpha$ and $\beta$ for which its analysis is tractable. More important, we establish positivity and rigidity results results for the energy in those cases. These results are intimately connected with the existence of metrics with positive constant $Q$-curvature

Our aim here is to analyse the particular case of stationary solutions of the fourth-order field equations parameterized by $2\alpha+\beta=0$. Recall that globally hyperbolic stationary space-times are manifolds of the form $V = M\times\mathbb{R}$ endowed with a Lorentzian metric that can be written as
\begin{align}
\bar{g}=-N^2dt^2 + \tilde{g},
\end{align}
where $N:M\mapsto \mathbb{R}^{+}$ is the \textit{lapse} function and $\tilde{g}\in\Gamma(T^0_2M)$ restricts to a Riemannian metric $g$ on each $t=constant$ hypersurface. In this setting, the appropriate notion of energy associated to the action $S$ and the corresponding field equations $A_{\bar{g}}=0$ becomes
\begin{align}\label{4thenergy-static.2}
\begin{split}
\mathcal{E}_{\alpha}(g)\doteq - \alpha\lim_{r\rightarrow\infty}   \int_{S^{n-1}_r} \left( \partial_{j}\partial_{i}\partial_{i}g_{aa} - \partial_{j}\partial_{u}\partial_{i}g_{u i}\right)\nu^{j}d\omega_{r}   \end{split}.
\end{align}
In this expression, we are assuming that $(M,g)$ is an AE manifold. This means that for a given compact set $K$ the asymptotic region $M\backslash K$ is diffeomorphic to $\mathbb{R}^n\backslash\overline{B_1(0)}$ and it is foliated by topological $(n-1)$-dimensional spheres $S^{n-1}_r$, whose Euclidean volume element is denoted in the expression above by $d\omega_r$. There, $\nu$ stands for the Euclidean unit normal field to these spheres.

The nature of (\ref{4thenergy-static.2}) as a conserved quantity in the context of higher-order gravitational theories make $\mathcal{E}_{\alpha}(g)$ a very good fourth-order analogue to the Arnowitt, Deser and Misner (ADM) energy in the context of GR. Let us recall that the total energy of an isolated gravitational system in GR, whose initial data is modelled as an AE manifold, is given by (see \cite{ADM})
\begin{align}\label{ADMenergy}
E_{ADM}(g)&=c(n)\lim_{r\rightarrow\infty}   \int_{S^{n-1}_r} \left( \partial_{i}g_{j i} - \partial_{j}g_{ii}\right)\nu^{j}d\omega_{r},
\end{align} 
where $c(n)$ stands for a dimensional constant (see \cite{Bartnik} for the detailed analytical properties of (\ref{ADMenergy})). This energy $E_{ADM}$ has had a huge impact both within GR and in geometric analysis. Most notoriously, it was the work of R. Schoen in \cite{Schoen1} that elucidated the role that the ADM energy plays in the final resolution of the Yamabe problem (see also \cite{Lee-Parker} for a review on this topic). In particular, in order to solve the Yamabe problem in the positive Yamabe class, in dimensions 3,4 and 5 or in locally conformally flat manifolds, Schoen noticed that it was enough to prove that an appropriate constant appearing in the expansion of the Green function $G_{L_g}$ associated to the conformal Laplacian $L_g$ was non-negative and that the zero case implied rigidity with the round sphere. Furthermore, it was pointed out that this constant was precisely the ADM energy of the AE manifold obtained via an stereographic projection. Therefore, the proof of the positive energy theorem in GR actually underlies the proof of the Yamabe problem in these cases. This beautiful relation shows what a fundamental result the positive energy theorem actually is within geometric analysis, being a cornerstone in the resolution of the Yamabe problem. Since then, the ADM energy has influenced many other constructions within geometric analysis and mathematical GR, which are not necessarily concerned with the Yamabe problem. For instance, rigidity phenomena associated to positive scalar curvature \cite{Carlotto1,Carlotto2,Carlotto3}, isoperimetric problems on AE manifolds \cite{Eichmair1,Eichmair2,Yau1}, geometric foliations and center of mass constructions \cite{Eichmair2,Huang,Yau1,Nerz} and even gluing constructions \cite{Carlotto3} (for nice reviews of many of these topics, see \cite{Carlotto4,Lee}).

In view of the above paragraph, we consider that the analysis of the positive energy theorem associated to (\ref{4thenergy-static.2}) stands as a highly well-motivated problem both for the development of a program devoted to the mathematical analysis of fourth order theories of gravity, as well as a tool which can potentially play a fundamental role in several fourth order geometric problems. Thus, the main objective of this paper is the establishment of such a positive energy theorem, and, afterwards, we will show how it underlies some fundamental problems in geometric analysis. 

Since we will be concerned with positivity issues related to $\mathcal{E}_{\alpha}(g)$, we must actually fix the sign of $\alpha$ a priori. For reasons that will become apparent through this paper, we will fix $\alpha=-1$ and define $\mathcal{E}(g)\doteq \mathcal{E}_{-1}(g)$. Let us notice that before embarking on the proof of the positive energy theorem, we should first analyse under what conditions $\mathcal{E}(g)$ is actually well-defined and prove that it is actually an intrinsic geometric object within a suitable class of AE-manifolds. This will be done in Proposition \ref{Static.1}, Proposition \ref{Geometricmass} and Theorem \ref{Uniqueness.1}. Once this is done, we will first explore appropriate geometric conditions that could in principle provide a rigidity statement in the critical cases. Explicitly, such conditions will involve a $Q$-curvature condition of the form $Q_g\geq 0$ (see Proposition \ref{Maxprinciple.1} and Theorem \ref{PEthm3}). Let us recall that given a Riemannian manifold $(M^n,g)$, $n\geq 3$, its $Q$-curvature is defined by\footnote{See Appendix B for details on our conventions on $Q$-curvature.}
\begin{align*}
Q_g\doteq-\frac{1}{2(n-1)}\Delta_gR_g - \frac{2}{(n-2)^2}|\mathrm{Ric}_g|^2_g + \frac{n^3-4n^2+16n-16}{8(n-1)^2(n-2)^2}R_g^2.
\end{align*}

In this context, we will establish the following theorem:
\begin{theorem}[Positive Energy]\label{PETHM}
Let $(M^n,g)$ be an $n$-dimensional AE manifold, with $n\geq 3$, which satisfies the decaying conditions: (i) $g_{ij}-\delta_{ij}=O_4(r^{-\tau})$, with $\tau>\max\{0,\frac{n-4}{2}\}$,  in some coordinate system associated to an structure of infinity; (ii) $Q_g\in L^1(M,dV_g)$, and such that $Y([g])>0$ and $Q_g\geq 0$. Then, the fourth order energy $\mathcal{E}(g)$ is non-negative and $\mathcal{E}(g)=0$ if and only if $(M,g)$ is isometric to $(\mathbb{R}^n,\delta)$.
\end{theorem}

The proof of this statement will be the content of Theorem \ref{PEthm} (see also theorems \ref{PEthm2} and \ref{PEthm3}). Let us notice that some hypotheses in the above theorem can be relaxed while keeping important results. In particular, the positivity and rigidity statements are somehow decoupled. This implies that under hypotheses (i)-(ii) and $R_g>0$ \textit{at infinity}, it follows that $\mathcal{E}(g)\geq 0$ (see Theorem \ref{PEthm2}), while under hypotheses (i)-(ii), $Q_g\geq 0$ and $g$ Yamabe positive, if $\mathcal{E}(g)=0$, then the rigidity statement follows (see Theorem \ref{PEthm3}). Furthermore, let us notice that the borderline case of $n=4$ is special, since it follows that any 4-dimensional AE manifold satisfies the decaying conditions (i)-(ii) and, in fact, under these conditions $\mathcal{E}(g)=0$. These last comments also apply to the case $n=3$, which can be checked explicitly quite straightforwardly. This imposes restrictions on the positive curvature conditions that these manifolds can admit, providing us with the following corollaries.

\begin{cor*}
Any n-dimensional AE-Riemannian manifold $(M^n,g)$, with $n\in\{3,4\}$, such that $Q_g\geq 0$ and $Y([g])> 0$, is isometric to $(\R^n,\delta)$.
\end{cor*}

Let us notice that Yamabe positive  AE manifolds in low dimensions (specially three) are very natural objects for instance in GR, since maximal vacuum initial data for isolated systems belong to this class of manifolds. Therefore, the above corollary, in some sense, can be used to separate the trivial solutions via an appeal to their $Q$-curvature.

The following result, which in particular gives a conformally invariant statement, is also a direct consequence of the analysis related to the above positive energy theorem for $\mathcal{E}(g)$ (see Proposition \ref{propc} for more details).

\begin{cor*}
Let $(M,g)$ be an asymptotically Euclidean four manifold with $\kappa_g =\int_M Q_g \, dv_g \geq 0$ and $Y([g])> 0$, then $(M,g)$ is conformal to the euclidean $\R^4$.
\end{cor*}

After presenting the above results concerning the positive energy theorem associated to $\mathcal{E}(g)$, we will make contact with important problems in geometric analysis. In particular, our aim is to show that that there is a very clear parallel in the role played by the ADM positive energy theorems of Schoen-Yau \cite{SY1,SY2,SY3} in the resolution of the Yamabe problem, to the role Theorem \ref{PETHM} plays in the resolution of the positive $Q$-curvature Yamabe-type problems which are known up to date. This problem concerns finding a conformal deformation of a closed Riemannian manifold $(M,g)$ so that the resulting metric has constant $Q$-curvature. There has been great progress in this program in recent times and the analysis depends on whether $n\geq 5$ or $n=3,4$. In particular, for the most updated resolutions of this problem to our knowledge in dimension four see \cite{Malchiodi2,Gursky,Chang-Yang} and references therein, while for $n\geq 5$, with different degrees of generality, let us draw the reader's attention to \cite{hangyang1,Malchiodi1,Qing}. Although we will not be concerned with the three dimensional case, let us point out that the problem has been addressed in this case in \cite{hangyang2,hangyang3}. 

Let us now focus on the results of \cite{hangyang1}. There, the authors prove that given a closed manifold $(M,g)$ of dimension $n\geq 5$ which is Yamabe non-negative and satisfies $Q_g\geq 0$ not identically zero,\footnote{In this cases we will say that $Q_g$ is \textbf{semi-positive}.} there is a conformal deformation to constant positive $Q$-curvature. In particular, this is done in Theorem 4 therein where in the cases of dimensions $n=5,6,7$ or $n\geq 5$ and locally conformally flat, there is a parallel to Schoen's resolution of the Yamabe problem. That is, the problem can be reduced to showing that a certain coefficient in the Green function expansion of the Paneitz operator near a pole is non-negative. In analogy to Schoen's ideas, this constant has been called the mass of the Paneitz operator  (whenever defined) and its positivity has been analysed first by Humbert-Raulot in the locally conformally flat case \cite{Raulot}, then by Gursky-Malchiodi who incorporated the cases $5\leq n\leq 7$ without this last restriction, and finally by Hang-Yang \cite{hangyang2}, who weakened the hypotheses on the scalar curvature imposed in \cite{Malchiodi1} and achieved the final form of this results which was applied in Theorem 4 of \cite{hangyang1}. Along these lines, let us also draw the reader's attention to an unpublished work by B. Michel \cite{BM}, who also analysed these kinds of positive mass theorems associated to the Green function of the Paneitz operator previously and obtained related results through similar methods to those of \cite{Raulot,Malchiodi1,hangyang2}.

In Section 2.2 we will show that exactly in dimensions $5\leq n\leq 7$ or if $(M,g)$ has a point $p$ around which it is conformally flat, then the mass of the Green function $G_{P_g}$ associated to the Paneitz operator $P_g$ is positively proportional to the energy $\mathcal{E}(\hat{g})$ of the asymptotically euclidean manifold $(\hat{M}=M\backslash\{p\},\hat{g}=G^{\frac{4}{n-4}}_{P_g}g)$ obtained via a stereographic projection. Then, we show that the following $Q$-curvature positive mass theorem follows from Theorem \ref{PETHM}:

\begin{theorem}[$Q$-curvature Positive Mass \cite{Malchiodi1,hangyang1,Raulot}]
Let $(M,g)$ be a closed $n$-dimensional Riemannian manifold, with $ 5\leq n\leq 7$  or $n\geq 8$ and  locally conformally flat around some point $p\in M$. If $Y([g])\geq 0$ and $(M,g)$ admits a conformal metric with semi-positive $Q$-curvature, then the mass of $G_{P_g}$ at $p$ is non-negative and vanishes if and only if $(M,g)$ is conformal to the standard sphere.
\end{theorem} 

The above theorem is exactly the positive mass theorem used by Hang-Yang in \cite{hangyang1} to solve the $Q$-curvature prescription problem in dimensions $5\leq n\leq 7$ or $n\geq 5$ and $M$ locally conformally flat around a point. This highlights the potential parallel of energy $\mathcal{E}(g)$ in the analysis of fourth-order problems to the role of the ADM energy in classical second order geometric problems. 

Finally, along the lines of the remarks of the previous paragraphs, we will analyse the critical four-dimensional case and provide an independent proof a Theorem B in \cite{Gursky} appealing to the techniques derived in this paper. Concretely, the following theorem follows from our analysis:

\begin{theorem}[Gursky]
Let $(M^4,g)$ a $4$-dimensional manifold with $Y([g])\geq 0$, then $\kappa_g\leq 16\pi^2$ with equality holding iff $(M^4,g)$ is conformal to the standard sphere.
\end{theorem}

%The First part was already proved in \cite{Gursky} and the second part with the weaker hypotheses that $\kappa_g > 4\pi^2 \chi(M^4)$ in \cite{CGY} but with only a topological equivalence and \textit{ a priori } not a conformal equivalence.
We would like to highlight that, in this case, our techniques make contact with the positive mass theorem of 2-dimensional manifolds, as presented in \cite{Lee}.

Finally, let us comment that, with the aim of delivering a self-contained presentation but trying to avoid a long introduction to preliminary results concerning analysis on AE-manifolds, $Q$-curvature analysis and constructions concerning conformal normal coordinates, which are things very well-known for experts in each of these fields, we have compiled the main results which will be used in this paper is the Appendices A,B and C, where detailed references can be found.

It has been brought to our attention by Marc Herzlich that one of his students, Benoît Michel, has previously analysed a similar notion of energy, see \cite{BM}.

\bigskip
{\bf Acknowledgments:} The authors would like to thank the CAPES-COFECUB, CAPES-PNPD and ANR (ANR- 18-CE40-002) for their financial support. Also, we would like thank the comments, suggestions and critics made by an anonymous referee to this paper, which have helped us improve the presentation and content of the paper.

\section{Preliminaries}%\todo[author=Paul, color=blue!50!white]{ probably better in the appendix, but we should also defined the convention such as $\omega_n$ once for all}

In this section we will collect some definitions and notational conventions which will be used along the paper. 

\bigskip
\noindent\textbf{Some notational conventions:}
\begin{itemize}
\item $(M^n,g)$ will denote an $n$-dimensional Riemannian manifold.
\item Given a Riemannian manifold $(M,g)$, we will denote by $\Delta_gu=g^{ij}\nabla_i\nabla_ju$ the negative Laplacian.
\item Given a Riemannian manifold $(M,g)$ we denote by $\nabla$ its Riemannian connection and define the curvature tensor by
\begin{align*}
R(V,W)Z=\nabla_V\nabla_WZ - \nabla_W\nabla_VZ - \nabla_{[V,W]}Z, \;\; V,W,Z\in \Gamma(TM),
\end{align*}
and we label its components in some coordinate system $\{x^i\}_{i=1}^n$ via $R^{i}_{jkl}=dx^i\left( R(\partial_k,\partial_l)\partial_j \right)$. Then, the Ricci tensor is defined locally via the contraction $R_{jl}=R^{k}_{jkl}$.
\item Given a tensor field $T$, indices of $\nabla^k T$ resulting from covariant differentiation will be separated by a comma. That is, if $T_{ij}$ are the components of a $2$-tensor, then we denote the components of $\nabla^2T$ by $T_{ij,kl}$. 
\item Given $\Sigma\hookrightarrow M$ an embedded compact hypersuface in $(M,g)$, we will denote by $\nu$ the outward pointing $g$-unit normal to $\Sigma$.
\item Given $\Sigma\hookrightarrow \mathbb{R}^n$ an embedded compact hypersuface, we will denote by $\nu^{e}$ Euclidean outward pointing unit normal to $\Sigma$.
\item $\omega_{n-1}$ will denote the volume of the unit sphere $\mathbb{S}^{n-1}\hookrightarrow \mathbb{R}^{n}$ and $d\omega_{n-1}$ is its canonical volume measure.
\item Given a compact manifold $(M,g)$, the conformal class of $g$ is $\{ u g\; ; \; u\in C^{\infty}(M, \R_+^*)\}$, denoted $[g]$.
\item Given a compact manifold $(M,g)$, its Yamabe invariant is defined as
\[ Y([g])= \inf_{\tilde{g}} \frac{\int_M R_{\tilde{g}} \, dv_{\tilde{g}}}{\mathrm{vol}(M)^\frac{n-2}{2}} .\]  
\item Given a function $f\in C^{k}(\mathbb{R}^n)$, we say that $f=O_k(|x|^{\tau})$ for some $k\geq 0$ and $\tau\in \mathbb{R}$ if for any multi-index $\alpha$ with $|\alpha|\leq k$, the functions $\partial^{\alpha}f= O(|x|^{\tau-|\alpha|})$ either as $|x|\rightarrow\infty$ or $|x|\rightarrow 0$, depending on the context. %We will also use the analogous little-o notation as well.
\end{itemize}

\begin{defn}[Weighted spaces] Given $x\in \mathbb{R}^n$, let us define $r(x)\doteq |x|$, $\sigma(x)\doteq(1+r^2(x))^\frac{1}{2}$ and consider $\delta \in \R$. Then, we set
    \begin{itemize}
    \item $L^p_\delta(\R^n)=\left\{  u\in L^p_{loc}\, \vert\, \int_{\R^n} \vert u\vert^p \sigma^{-\delta p- n}\, dx <+\infty \right\}$, we equip this set with the norm $\Vert u\Vert_{p,\delta}^p =\int_{\R^n} \vert u\vert^p \sigma^{-\delta p- n}\, dx$.
    \item $L^\infty_\delta(\R^n)=\left\{  u\in L^\infty_{loc}\, \vert\, \sup_{\R^n} \vert u\vert \sigma^{-\delta} <+\infty \right\}$, we equip this set with the norm $\displaystyle \Vert u\Vert_{\infty,\delta} =\sup_{\R^n} \vert u\vert \sigma^{-\delta}$.
    \item $W^{k,p}_\delta(\R^n)=\left\{  u\in W^{k,p}_{loc}\, \vert\, \sum_{j=0}^k \Vert d^j u\Vert_{p,\delta-j}  <+\infty \right\}$, we equip this set with the norm $\Vert u\Vert_{k,p,\delta}=\sum_{j=0}^k \Vert d^j u\Vert_{p,\delta-j} $.
    \end{itemize}
\end{defn}

In particular, by Sobolev embedding, when  $kp>n$ if $u \in W^{k,p}_\delta(\R^n)$ then $u =O(r^\delta)$ at infinity. See \cite{Bartnik} for more details about weighted spaces.

Then we introduce the notion of asymptotically flat manifolds.
\begin{defn}[AE manifolds]
    A complete Riemannian manifold $(M,g)$ with $g\in W^{k,q}_{loc}(M)$ for some $k\geq 1$ and $q>n$ is said to be asymptotically Euclidean (with one end) if there exists a compact set $K\subset M$ and a diffeomorphism  $\phi :M\setminus K \rightarrow \R^n\setminus \overline{B_1(0)}$ such that
    \begin{itemize}
    \item $\phi_{*}(g)$ is a uniformly positive defined metric, i.e. there exists $\lambda>1$ such that
    $$\frac{1}{\lambda} \vert \xi\vert^2 \leq g_{ij}(x)\xi^i\xi^j \leq \lambda \vert \xi\vert^2 \; \forall x\in \R^n \setminus B_1 \; \forall \xi \in \R^n.$$
    \item $$\phi_*(g)_{ij}-\delta_{ij} \in W^{k,q}_{-\tau} (\R^n\setminus B(0,1))$$
    for some $\tau >0$ called the decreasing rate.
    \end{itemize} 
\end{defn} 
    {\bf Important remark: } In this chart we defined $\sigma$ and we remark that the definition of $L^q_\delta$ is independent of the chart. But $W^{k,q}_{\delta}$ depends on the chart $\phi$, since the partial derivatives will depend on the choice of coordinates. It will be denoted $W^{k,q}_{\delta}(\phi)$. In the rest of the section, we will call such a chart {\bf a structure at infinity}. Once the structure at infinity is chosen, we naturally extend the definition of $W^{k,q}_{\delta}(\R^n\setminus B(0,1))$ to $W^{k,q}_{\delta}(M)$  on the compact part, noticing that all choices of charts in the compact region define equivalent norms.\\

In the following definition, which will be used in the core of this paper, we will restrict to AE manifold which posses more regularity than that of the general definition given above.
\begin{defn}
We will say that a (smooth) AE manifold $(M,g)$ is of order $\tau>0$ with respect to some structure at infinity $\Phi:M\backslash K\mapsto \mathbb{R}^n\backslash\bar{B}$, if, in such coordinates, $\Phi_{*}g_{ij}-\delta_{ij}=O_4(|x|^{-\tau})$.
\end{defn}

Of course we can define asymptotically flat structures with multiple ends. But since analysis phenomena are determined by the behaviour at infinity, it is very easy to isolate each end and to consider that there is only one.

\section{The positive energy theorem}

The main result of this section will be a positive energy theorem related to (\ref{4thenergy-static.2}) with its corresponding rigidity statement. But, before this, we will begin by analysing geometric conditions under which (\ref{4thenergy-static.2}) is well-defined. In the following proposition, we will establish such geometric criteria. Let us first set 
$$\tau_n = \left\{ \begin{array}{ll} 0 &\hbox{ if } n=3,4 \\  \frac{n}{2}-2 &\hbox{ if } n\geq 5\end{array} \right. .$$

\begin{prop}\label{Static.1}
Let $(M^n,g)$ be an AE manifold of dimension $n\geq 3$ satisfying the following conditions
\begin{enumerate}[label=\roman*)]
\item There are end rectangular coordinates, given by a structure of infinity $\Phi$, where $g_{ij}=\delta_{ij}+O_4(r^{-\tau})$, where $\tau>\tau_n$;
\item The $Q$-curvature of $g$ is in $L^1(M,dV_{g})$.
\end{enumerate}
\noindent Then, given an exhaustion of $M$ by compact sets $\Omega_k$ such that $S_k\doteq \Phi(\partial \Omega_k)$ are smooth connected $(n-1)$-dimensional manifolds without boundary in $\mathbb{R}^n$ satisfying
\begin{align}
\begin{split}
R_k\doteq \inf\{|x|: x\in S_k\}\xrightarrow[k\rightarrow\infty]{} \infty,\\
R_k^{n-1}\mathrm{area}(S_k) \text{ is bounded as } k\rightarrow\infty,
\end{split}
\end{align}
the limit 
\begin{align}\label{Energy.1}
\mathcal{E}^{(\Phi)}(g)= \lim_{k\rightarrow\infty}   \int_{S_k} \left( \partial_{j}\partial_{i}\partial_{i}g_{aa} - \partial_{j}\partial_{u}\partial_{i}g_{u i}\right)\nu^{j}dS,
\end{align}
exists and is independent of the sequence of $\{S_k\}$ used to compute it. 
\end{prop}

%\begin{remark}
%Recall that given a Riemannian metric $g$ its $Q$-curvature is defined by
%\begin{align}\label{Q-curv}
%Q_g=-\frac{1}{2(n-1)}\Delta_gR_g - \frac{2}{(n-2)^2}|\mathrm{Ric}_g|^2_g + \frac{n^3-4n^2+16n-16}{8(n-1)^2(n-2)^2}R^2_g.
%\end{align} 
%\end{remark}
\begin{proof}
From the decaying conditions and the familiar expression
\begin{align*}
R_g&=\partial_i\partial_jg_{ij} - \partial_j\partial_jg_{ii} + O_2((g-\delta)\partial^2g) + O_2(\left(\partial g\right)^2),\\
&=\partial_i\partial_jg_{ij} - \partial_j\partial_jg_{ii} + O_2(r^{-2\tau-2})
\end{align*}
we get
\begin{align*}
\Delta_gR_g&=g^{ab}\nabla_a\nabla_bR_g=g^{ab}\left(\partial_a\partial_bR_g + \Gamma^c_{ab}\partial_cR_g \right),\\
&=g^{ab}\partial_a\partial_b\Big(\partial_i\partial_jg_{ij} - \partial_j\partial_jg_{ii}\Big) + g^{ab}\partial_a\partial_b\Big(O_2(r^{-2\tau-2})\Big) + g^{ab}\Gamma^c_{ab}\partial_c\left(\partial_i\partial_jg_{ij} - \partial_j\partial_jg_{jj}\right) \\
&+ g^{ab}\Gamma^c_{ab}\partial_c\left(O_2\left(r^{-2\tau-2}\right)\right),\\
&=g^{ab}\partial_a\partial_b\Big(\partial_i\partial_jg_{ij} - \partial_j\partial_jg_{ii}\Big) + \left(g^{ab}-\delta^{ab}\right)O(r^{-2\tau-4}) + O(r^{-2\tau-4}) \\
&+ O_3\left((g-\delta)\partial g\right))O_1(r^{-\tau-3}) +  O_3(\partial g)O_1(r^{-\tau-3}) + O_3\left((g-\delta)\partial g\right))O_1(r^{-2\tau-3}) \\
&+  O_3(\partial g)O_1(r^{-2\tau-3}),\\
%&=\partial_a\partial_a\Big(\partial_i\partial_jg_{ij} - \partial_j\partial_jg_{ii}\Big) + O(g-\delta)O(r^{-\tau-4}) + O(r^{-3\tau-4}) + O(r^{-2\tau-4}) \\
%&+ O\left(r^{-3\tau-4}\right) +  O_1(r^{-2\tau-4}) + O(r^{-4\tau-4}) +  O(r^{-3\tau-4}),\\
&=\partial_a\partial_a\Big(\partial_i\partial_jg_{ij} - \partial_j\partial_jg_{ii}\Big) + O(r^{-2\tau-4}). 
\end{align*}
In particular, denoting by $D_k$ the annular region between $S_{k+1}$ and $S_{k}$, we get that
\begin{align*}
\int_{D_k}\partial_j\partial_{j}\left(\partial_{u}\partial_{i}g_{u i} - \partial_{u}\partial_{u}g_{ii}\right)dV_e&=\int_{S_{k+1}}(\partial_{j}\partial_{u}\partial_{i}g_{u i} - \partial_{j}\partial_{u}\partial_{u}g_{ii})\nu^e_jdS^{n-1} \\
&- \int_{S_{k}}(\partial_{j}\partial_{u}\partial_{i}g_{u i} - \partial_{j}\partial_{u}\partial_{u}g_{ii})\nu^e_jdS^{n-1},
\end{align*}
where, following our convention, $\nu^{e}$ denotes the outward pointing Euclidean unit normal to each hypersurface. Therefore, if the left-hand side is integrable over $\mathbb{R}^n\backslash K$, then the above boundary integrals form a Cauchy sequence as $r\rightarrow\infty$ and therefore (\ref{Energy.1}) is well defined. But from the above computations, we see that this can be reduced to $\Delta_gR_g\in L^1(M)$ and $2(\tau+2)>n$, that is $\tau>\tau_n$. In particular, notice that under our decaying conditions $R_g,\mathrm{Ric}_g=O(r^{-\tau-2})$ near infinity, implying that
\begin{align*}
Q_g=-\frac{1}{2(n-1)}\Delta_gR_g + O(r^{-2\tau - 4}),
\end{align*}
therefore the leading order is carried on the first term and $\Delta_gR_g\in L^1(M)$ can be replaced by $Q_g\in L^1(M)$.
\end{proof}

\bigskip

Now, we intend to show that (\ref{4thenergy-static.2}) is a geometric object, independent of the structure of infinity we use. With this in mind, let us start by rewriting (\ref{4thenergy-static.2}) in more geometric fashion, which will be proved to be more useful for its analysis.
\begin{prop}\label{Geometricmass}
Let $(M,g)$ be an AE manifold which satisfies the decaying conditions i) and ii) of Proposition \ref{Static.1}. Then, we can rewrite the energy (\ref{4thenergy-static.2}) as
\begin{align}\label{Static.2}
\mathcal{E}(g)&=-\lim_{r\rightarrow\infty}\int_{S^{n-1}_1}\partial_r R_g r^{n-1}d\omega_{n-1}.
\end{align}
\end{prop}
\begin{proof}
%With this in mind, from the above, let us assume as above that asymptotically $g_{ij}=\delta_{ij}+O_4(r^{-\tau})$, with $\tau>\frac{n}{2}-2$, $n\geq 4$, and $Q_g\in L^1(M)$. 
Let us denote by $\nu$ the outward $g$-unit normal to $S^{n-1}_r\hookrightarrow \mathbb{R}^n$, with $r$ sufficiently large, and by $\nu^e$ the outward euclidean unit normal to the same sphere. Then, it follows that $|\nu^e-\nu|_g=O(r^{-\tau})$. Thus, we find that
\begin{align*}
g(\nabla R_g,\nu)&=g(\nabla R_g,\nu^e)+ g(\nabla R_g,\nu-\nu^e),\\
&=\partial_iR_g\nu^e_i + \underbrace{(g-\delta)_{ij}\nabla^iR_g(\nu^e)^j}_{O_1(r^{-2\tau - 3})} + \underbrace{\partial_iR_g(\nu-\nu^e)_i}_{O_1(r^{-2\tau-3})} + \underbrace{(g-\delta)_{ij}\nabla^iR_g(\nu-\nu^e)^j}_{O_1(r^{-3\tau - 3})}.
\end{align*}
Also, we already know that $R_g=\partial_i\partial_jg_{ij} - \partial_j\partial_jg_{ii} + O_2(r^{-2\tau-2})$, which implies
\begin{align*}
g(\nabla R_g,\nu)&=\partial_i\left(\partial_k\partial_jg_{kj} - \partial_j\partial_jg_{kk}\right)\frac{x^i}{r} +  O_1(r^{-2\tau-3}) %+ \partial_i\left(\partial_k\partial_jg_{kj} - \partial_j\partial_jg_{kk}\right)(\nu-\nu^e)_i \\
%&+ \underbrace{O_2(r^{-2\tau-3})(\nu-\nu^e)_i}_{O_2(r^{-\frac{5}{2}\tau-3})} + O_1(r^{-2\tau - 3}),\\
%&=\underbrace{\partial_i\left(\partial_k\partial_jg_{kj} - \partial_j\partial_jg_{kk}\right)}_{O_1(r^{-\tau-3})}\frac{x^i}{r} + \underbrace{\partial_i\left(\partial_k\partial_jg_{kj} - \partial_j\partial_jg_{kk}\right)(\nu-\nu^e)_i}_{O_1(r^{-\frac{3}{2}\tau -3})} + O_1(r^{-2\tau-3}).
\end{align*}
Thus, 
\begin{align*}
\int_{S^{n-1}_1}g(\nabla R_g,\nu)r^{n-1}d\omega_{n-1}&=\int_{S^{n-1}_1}\partial_i\left(\partial_k\partial_jg_{kj} - \partial_j\partial_jg_{kk}\right)\frac{x^i}{r}r^{n-1}d\omega_{n-1} + O_1(r^{n-2\tau-4}).
\end{align*}
Now, from our decaying condition $\tau>\tau_n$, we see that $n-2\tau-4<0$, which implies that in the limit $r\rightarrow\infty$ the last term vanishes. 
%Now, \textbf{let us assume that $R_g=O_2(r^{-\tau^{*}})$ for any $\tau^{*}<n-2$.} If this holds, then notice that our decaying conditions imply that $-2\tau-2<-n+2<-\tau^{*}$, which shows that
%\begin{align*}
%R_g=\partial_i\partial_jg_{ij} - \partial_j\partial_jg_{ii} + O_2(r^{-2\tau-2})=O_2(r^{-\tau^{*}})%=O_2(r^{-\tau^{*}}),
%\end{align*}
%that is, 
%\begin{align*}
%\partial_i\partial_jg_{ij} - \partial_j\partial_jg_{ii}=O_2(r^{-\tau^{*}})%=O_2(r^{-2\tau-2}).
%\end{align*}
%Putting this together with $\nu-\nu^e=O(r^{-\frac{\tau}{2}})$, implies that
%\begin{align*}
%\Big|\int_{S^{n-1}_1}\partial_i\left(\partial_k\partial_jg_{kj} - \partial_j\partial_jg_{kk}\right)(\nu-\nu^e)_ir^{n-1}d\omega_{n-1}  \Big|&\lesssim r^{n-1-\frac{\tau}{2}} \int_{S^{n-1}_1}|\partial_i\left(\partial_k\partial_jg_{kj} - \partial_j\partial_jg_{kk}\right)|d\omega_{n-1},\\
%&\lesssim r^{n-1-\frac{\tau}{2}}r^{-\tau^{*}-1}=r^{n-\tau^*-2}r^{-\frac{\tau}{2}}%&\xrightarrow[r\rightarrow\infty]{} 0,
%&\lesssim r^{n-1-\frac{\tau}{2}}r^{-n+1}\xrightarrow[r\rightarrow\infty]{} 0
%\end{align*}
%Since $\tau^{*}$ can be taken as close to $n-2$ as we want, pick $n-2-\frac{\tau}{4}<\tau^{*}<n-2$, so that
%\begin{align*}
%\Big|\int_{S^{n-1}_1}\partial_i\left(\partial_k\partial_jg_{kj} - \partial_j\partial_jg_{kk}\right)(\nu-\nu^e)_ir^{n-1}d\omega_{n-1}  \Big|&\lesssim r^{n-\frac{\tau}{4} -2 -\tau^*}r^{-\frac{\tau}{4}}\xrightarrow[r\rightarrow\infty]{} 0.
%\end{align*}
Therefore, under these conditions
\begin{align*}
\lim_{r\rightarrow\infty}\int_{S^{n-1}_1}g(\nabla R_g,\nu)r^{n-1}d\omega_{n-1}&=\lim_{r\rightarrow\infty}\int_{S^{n-1}_1}\partial_i\left(\partial_k\partial_jg_{kj} - \partial_j\partial_jg_{kk}\right)\nu^e_ir^{n-1}d\omega_{n-1}.
\end{align*}
That is, 
\begin{align}
\begin{split}
\mathcal{E}(g)&=-\lim_{r\rightarrow\infty}\int_{S^{n-1}_1}g(\nabla R_g,\nu)r^{n-1}d\omega_{n-1},\\
&=-\lim_{r\rightarrow\infty}\int_{S^{n-1}_1}\nabla R_g\cdot \nu^e r^{n-1}d\omega_{n-1},\\
&=-\lim_{r\rightarrow\infty}\int_{S^{n-1}_1}\partial_r R_g r^{n-1}d\omega_{n-1}.
\end{split}
\end{align}
\end{proof}

With the aid of the above two results, we can establish the following theorem, which establishes that $\mathcal{E}(g)$ is a geometric object within a suitable class of AE manifolds. 

\begin{thm}\label{Uniqueness.1}
Let $(\phi,x)$ and $(\psi,y)$ be two structures of infinity for the AE manifold $(M,g)$ which satisfies the conditions i) and ii) of the above Proposition with decay rates $\tau_1,\tau_2$ respectively, satisfying $\tau\doteq \min\{\tau_1,\tau_2 \}> \tau_n$. Then, the energies $\mathcal{E}^{(\phi)}(g)$ and $\mathcal{E}^{(\psi)}(g)$ are well-defined and equal. Moreover, the following coordinate independent identity follows:
\begin{align}\label{EnergyBM}
\mathcal{E}(g)=-\int_{M}\Delta_gR_gdV_g.
\end{align}
\end{thm}
\begin{proof}
Appealing to Proposition \ref{Static.1}, let us compute the energy $\mathcal{E}^{\phi}(g)$, associated to the structure of infinity given by $(\phi,x)$ using a sequence of spheres near infinity of radii $\{r_{k}\}_{k=1}^{\infty}$. From the above proposition, we have that
\begin{align*}
\mathcal{E}^{(\phi)}(g)=-\lim_{|x|\rightarrow\infty}\int_{S^{n-1}_1}\nabla R_g \lrcorner dV_e,
\end{align*}
where $dV_e=r^{n-1}dr\wedge d\omega_{n-1}$ stands for the canonical Euclidean volume form. Appealing to the decaying conditions of $g$, we know that $\sqrt{\mathrm{det}(g)}=1+O(r^{-\tau})$ and therefore
\begin{align*}
\int_{S^{n-1}_1}\nabla R_g \lrcorner dV_g=\int_{S^{n-1}_1}\nabla R_g \lrcorner dV_e + O(r^{n-2\tau-4}).
\end{align*}
Since $n-2\tau-4<0$ under our hypotheses, the above implies that
\begin{align*}
\mathcal{E}^{(\phi)}(g)=-\lim_{r\rightarrow\infty}\int_{S^{n-1}_r}\nabla R_g \lrcorner dV_g.
\end{align*}
Finally, since $\mathrm{div}_g(\nabla R_g)dV_g=d(\nabla R_g\lrcorner dV_g)$, we find that
\begin{align*}
\mathcal{E}^{(\phi)}(g)=-\lim_{r\rightarrow\infty}\int_{D_r}\Delta_gR_gdV_g=-\int_{M}\Delta_gR_gdV_g,
\end{align*}
where $D_r\hookrightarrow M$ is the inner region in $M$ with $\partial D_r=S_r^{n-1}$. In particular, the right-hand side of the above expression is independent of the coordinates used near infinity. Thus we can drop the reference to $(\phi,x)$ and claim that for any structure of infinity of $M$ were $g$ satisfies conditions (i)-(ii) of Proposition \ref{Static.1}, it holds that
\begin{align*}
\mathcal{E}(g)=-\int_{M}\Delta_gR_gdV_g.
\end{align*}
\end{proof}

\begin{remark}
Let us highlight that the form for $\mathcal{E}$ given by (\ref{EnergyBM}) allows us to make contact with the ideas of \cite{Michel2}, where the author introduced a general method to produce \emph{asymptotic charges} from a Riemannian functional on a complete manifold with a model asymptotic structure. In particular, (\ref{EnergyBM}) is an example of such a functional. Let us nevertheless stress that the definition of energy we are interested in is the one (a priori) defined by equation (\ref{4thenergy-static.2}), which is the quantity appearing as a conserved charge for the fourth order gravitational theories described in the introduction. In this sense, the connection between (\ref{4thenergy-static.2}) and (\ref{EnergyBM}) is only revealed after the work done above through Propositions \ref{Static.1}, \ref{Geometricmass} and Theorem \ref{Uniqueness.1}, after which the conditions necessary to achieve asymptotic invariance have already been shown explicitly. Such conditions correspond to a somewhat simpler and explicit formulation of the corresponding abstract conditions presented in \cite{Michel2}, which must be checked in a case by case basis.
\end{remark}

Besides establishing sufficient conditions for the energy to be a well-defined intrinsic geometric object, the above theorem provides us with an easy positive energy corollary.
\begin{coro}\label{PEthm.0}
Let $(M^n,g)$ be an $n$-dimensional AE manifold, with $n\geq 3$, which satisfies the decaying conditions i) and ii) of Proposition \ref{Static.1} and such that $\Delta_gR_g\leq 0$. Then, the fourth order energy $\mathcal{E}(g)$ is non-negative and $\mathcal{E}(g)=0$ if and only if $(M,g)$ is scalar flat.
\end{coro}

The positivity statement in the above corollary is self-evident. On the other hand, the rigidity statement follows from the injectivity of the Laplacian under our hypotheses (see (\ref{deltag})). 

Let us notice that the kind of rigidity that we get from the above corollary is quite weak. In fact, although it is well-known that topological obstructions exist, the set of AE manifolds with zero scalar curvature is in general quite large (see, for instance, the discussion after Theorem \ref{PEthm}). Therefore, in order to get a stronger rigidity statement, we will need to impose so other geometric condition. We will explore this in the next proposition and in the main theorem given below. The following proposition, which is an adaptation of a result in \cite{Malchiodi1} will be used to insure that $R_g>0$ as long as $g$ is not flat. 

\begin{prop}\label{Maxprinciple.1}
Let $(M^n,g)$ be an AE Riemannian manifold with $n\geq 3$ satisfying (1) $Q_g\geq 0$ and (2) $R_g\geq 0$. Then, either $R_g>0$ or $g$ must be flat.
\end{prop}
\begin{proof}
If $Q_g\geq 0$, then we find that
\begin{align*}
\Delta_gR_g - c_1(n)R^2_g \leq -  c_2(n)|\mathrm{Ric}_g|^2_g\leq 0,
\end{align*}
for some constants $c_1(n),c_2(n)>0$. From the the condition $R_g\geq 0$ we can appeal to Lemma 4 in \cite{Maxwell1} to conclude that if $R_g$ vanishes at a single point, then it must be identically zero. Thus, we already see that either $R_g>0$ or $R_g\equiv 0$. In the second case, we find that $0\leq Q_g=-c_2(n)|\mathrm{Ric}_g|^2_g\leq 0$, and thus $\mathrm{Ric}_g\equiv 0$,
which implies (from asymptotic flatness) that $g$ is flat. Indeed, since $\mathrm{Ric}_g \geq 0$, then by Bishop-Gromov theorem insures that $r\mapsto \frac{vol(B_g(p,r))}{\omega_rr^n}$ is non-increasing, hence it must be constant by asymptotic flatness, finally the equality case of Bishop-Gromov theorem implies that $g$ must be flat.
\end{proof}

%\todo[author=Rodrigo]{Do you think that a comment like this is good enough to give credit to the papers where ideas we use in the proof were created?}

\bigskip
We will now present the main result of this paper. We should highlight that, in particular, the proof of the rigidity statement in the following theorem follows ideas close to the proofs of the positive mass theorems associated to the Paneitz operator of \cite{Raulot,Malchiodi1,hangyang2}.

\begin{thm}\label{PEthm}
Let $(M^n,g)$ be an $n$-dimensional AE manifold, with $n\geq 3$, which satisfies the decaying conditions i) and ii) of Proposition \ref{Static.1} and such that $Q_g,R_g\geq 0$. Then, the fourth order energy $\mathcal{E}(g)$ is non-negative and $\mathcal{E}(g)=0$ if and only if $Q_g \equiv 0$ and $(M,g)$ is isometric to $(\mathbb{R}^n,\cdot)$.
\end{thm}
\begin{proof}
Let us first notice that we need only work with the case $Q_g\geq 0$ and $R_g>0$, since in the remaining case the result follows from Proposition \ref{Maxprinciple.1}. Thus, in what follows we assume that $Q_g\geq 0$ and $R_g >0$.

\medskip
\noindent\textbf{Proof of positivity}

Let us start by noticing that under our hypotheses we have $R_g=O_2(r^{-\tau -2})$ with $\tau>\tau_n$. Then, the energy exists and is independent of the sequence of spheres used to compute it. Furthermore, these decaying conditions imply that (\ref{Static.2}) also holds. We can rewrite
\begin{align*}
\int_{S^{n-1}_1}\partial_r R_g r^{n-1}d\omega_{n-1}&=\frac{\partial}{\partial r}\left(r^{n-1}\int_{S^{n-1}_1} R_g d\omega_{n-1}\right) - (n-1)r^{n-2}\int_{S^{n-1}_1}R_g d\omega_{n-1}.
\end{align*}
Define 
\begin{align}
h(r)\doteq r^{n-2}\int_{S^{n-1}_1}R_g d\omega_{n-1}.
\end{align}
Then,
\begin{align}\label{PE.1}
\int_{S^{n-1}_1}\partial_r R_g r^{n-1}d\omega_{n-1}&=\frac{\partial}{\partial r}\left(r h(r)\right) - (n-1)h(r).
\end{align}

By assumption we have that $h>0$ and from our decreasing assumption that there is some $r_0$ sufficiently large, so that $R_g(x)\leq Cr^{ -\tau-2 }(x)$ for all $x\in M$ such that $r(x)>r_0$. In particular, this last condition implies that $h(r)\leq C\omega_{n-1} r^{-\tau +n -4}$ for all $r>r_0$.

Suppose that there is some $r^{*}> r_0$ such that (\ref{PE.1}) is positive at $r^{*}$. Then, by continuity, there is some interval $(r^{*}-\epsilon,r^{*}+\epsilon)$ so that
\begin{align*}
rh'(r)  > (n-2)h(r),
\end{align*}
which implies
\begin{align*}
\log\left(\frac{h(r)}{h(r^{*})}\right)>(n-2)\log\left(\frac{r}{r^{*}}\right), \text{ for any } r\in (r^{*},r^{*}+\epsilon).
\end{align*}
That is $h(r)>\frac{h(r^{*})}{{r^{*}}^{n-2}}r^{n-2}$ for all $r\in (r^{*},r^{*}+\epsilon)$. But notice that also $h(r)\leq C\omega_{n-1} r^{-\tau+n-4}$ for all such $r$, which shows that
\begin{align*}
\frac{h(r^{*})}{{r^{*}}^{n-2}}r^{n-2}<h(r)\leq C\omega_{n-1} r^{-\tau+n-4}\text{ for any } r\in (r^{*},r^{*}+\epsilon).
\end{align*}
Since $r^{*}$ and $C$ are fixed, the above inequalities are clearly falsified for large enough $r$, since $r^{n-2}$ grows faster that $r^{n-4-\tau}$. This, in turn, implies that the interval $I=[r^{*},r^{*}+\epsilon)$ where (\ref{PE.1}) is positive cannot be extended indefinitely, since that would force us into the above inequalities being valid for arbitrary large $r>r^{*}$, presenting a contradiction. Therefore, this implies that there must be some $\epsilon_{max}$ where (\ref{PE.1}) is non-positive at $r^{*}+\epsilon_{max}$. That is, there must be some $r_1>r^{*}$ where
\begin{align}
r_1^{n-1}\int_{S^{n-1}_1}\partial_r R_g(r_1,\theta)d\omega_{n-1}(\theta)&=\left(\frac{\partial}{\partial r}\left(r h(r)\right) - (n-1)h(r)\right)\vert_{r=r_1}\leq 0.
\end{align}
The idea is now to repeat this procedure so as to select a sequence of spheres $\{S^{n-1}_{r_i}\}_{i=1}^{\infty}$ along which (\ref{PE.1}) is manifestly non-positive. %\footnote{In fact, being able to construct such a sequence would prove that if $\mathcal{E}_{\alpha}(g)\neq 0$, then no such $r^{*}$ can be found after a finite number of steps. If not, we could also pick another sequence $\{\tilde{S}^{n-1}_{r_i}\}_{i=1}^{\infty}$ over which $(\ref{PE.1})0$ is positive and therefore along this sequence $\mathcal{E}_{\alpha}(g)\leq 0$. Since the limit exists, then $\mathcal{E}_{\alpha}(g)=0$.} 
 The fact that such radii can be chosen under our hypotheses is a consequence of the above argument. Then
\begin{align*}
\mathcal{E}(g)_j= -r_j^{n-1}\int_{S_{r_j}}\partial_rR_g d\omega_{n-1}\geq 0 \text{ for all } j,
\end{align*}
which implies
\begin{align}\label{Positivenergy.1}
\mathcal{E}(g)=-\lim_{j\rightarrow\infty}\int_{S_{r_j}}\partial_rR_g r^{n-1}d\omega_{n-1}\geq 0.
\end{align}
%Therefore, if the limit $\mathbb{E}_0^{4th}(g)$ exits and is therefore independent of the sequence we use to compute it, we could take it along $\{S^{n-1}_{r_i}\}_{i=1}^{\infty}$ where it would be a convergent sequence of non-negative numbers, implying that $\mathbb{E}^{4th}_0(g)\geq 0$ under the assumptions $\tau>\frac{n}{2}-2$; $Q_g\in L^1(M)$; $R_g=O_2(r^{-n+2})$ and $R_g>0$.
\medskip
\noindent\textbf{Proof of rigidity}

Now, the idea is to analyze whether $\mathcal{E}(g)=0$ implies $g$-flatness. First we look for a conformal metric $\tilde{g}=u^{\frac{4}{n-2}} g$ such that $R_{\tilde{g}}\equiv 0$, that is,
\begin{align}\label{scalcurvdeformation}
L_g(u)=\Delta_{g}u-c_n R_{g}u=0,
\end{align}
where $c_n=\frac{n-2}{4(n-1)}$. Setting $u=1+\phi$, this is equivalent to solve
\begin{align}\label{scalcurvdeformation2}
    \Delta_{g}\phi-c_n R_{g}\phi= c_nR_g,
\end{align}
We know that $R_{g}=O_2(r^{-\tau-2})$, this implies that $R_{g}\in L^{p}_{-\delta-2}$ for any $\delta< \tau$ and any $p$. Thanks to corollary \ref{CLg}, we know  that $L_g:W^{2,p}_{\rho}\mapsto L^{p}_{\rho-2}$ is an isomorphism for $2-n <\rho<0$. Thus, choosing $p>n/2$, we find a (smooth) solution $\phi\in W^{2,p}_{-\delta}$ for any $0<\delta<\sigma\doteq \min\{\tau,n-2 \}$.\footnote{In the case $\tau>n-2$, we can achieve $\delta=\sigma$.}. By our choice of $p$, $\phi$ goes to zero at infinity, then, the fact that $R_g >0$ and the maximum principle, see lemma 2 and 4 \cite{Maxwell1}, insure that $u=1+\phi>0$.  
 
\medskip
\noindent\underline{I) \texorpdfstring{$n\not=4$}{TEXT}}

\medskip
Then, let $\Phi\doteq u^{-\frac{n-4}{n-2}}$; rewrite $g=\Phi^{\frac{4}{n-4}}\tilde{g}$ and notice that transformation rule for the Paneitz operator gives us
\begin{align*}
\frac{n-4}{2}\Phi^{\frac{n+4}{n-4}}Q_{g}=P_{\tilde{g}}\Phi.
\end{align*}
Also, since $R_{\tilde{g}}=0$, we have that $Q_{\tilde{g}}=-\frac{2}{(n-2)^2}|\mathrm{Ric}_{\tilde{g}}|^2_{\tilde{g}}$, and thus
\begin{align}
P_{\tilde{g}}\Phi=\Delta^2_{\tilde{g}}\Phi + \mathrm{div}_{\tilde{g}}\left(4S_{\tilde{g}}(\nabla\Phi,\cdot) \right) -\frac{(n-4)}{(n-2)^2}|\mathrm{Ric}_{\tilde{g}}|^2_{\tilde{g}}\Phi,
\end{align}
where, $S_{\tilde{g}}=\frac{1}{n-2}\mathrm{Ric}_{\tilde{g}}$ since $R_{\tilde{g}}=0$. Therefore, denoting by $D_r$ the bounded region in $M$ whose boundary is given by the sphere $S_r$ in the end of $M$, we see that
\begin{align}\label{Rigidity.1}
\int_{D_r}\frac{n-4}{2}\Phi^{\frac{n+4}{n-4}}Q_{g}+ \frac{(n-4)}{(n-2)^2}|\mathrm{Ric}_{\tilde{g}}|^2_{\tilde{g}}\Phi \, dv_{\tilde{g}} = \int_{S_r}\tilde{g}(\tilde{\nabla}\Delta_{\tilde{g}}\Phi,\tilde{\nu})d\omega_{\tilde{g}} + \frac{4}{n-2}\int_{S_r}\mathrm{Ric}_{\tilde{g}}(\nabla\Phi,\tilde{\nu}) d\omega_{\tilde{g}} 
\end{align} 
Let us now estimate the terms in the right-hand side. 

\medskip
\noindent \textbf{Estimates on $\mathrm{Ric}_{\tilde{g}}(\nabla\Phi,\tilde{\nu})$}

From the above analysis we find that $u=1+O_4(r^{-\delta})$ for any $\delta<\sigma=\min\{\tau,n-2\}$, and therefore
\begin{align*}
\tilde{g}_{ij}&=(1+O_4(r^{-\delta}))(\delta_{ij}+O_4(r^{-\tau})),\\
&=\delta_{ij}+ O_4(r^{-\delta}).
\end{align*}
Then, $\mathrm{Ric}_{\tilde{g}}=O_2(r^{-\delta-2})$ and also
\begin{align*}
\nabla\Phi=-\frac{n-4}{n-2}u^{-\frac{n-4}{n-2}-1}\nabla u=O_3(r^{-\delta-1}).
\end{align*}
Which implies
\begin{align*}
\mathrm{Ric}_{\tilde{g}}(\nabla\Phi,\tilde{\nu})&=\mathrm{Ric}_{\tilde{g}}(\nabla\Phi,\nu^{e}) + {\mathrm{Ric}_{\tilde{g}}}_{ij}\nabla^i\Phi (\nu^{e}-\tilde{\nu})^j=O_2(r^{-2\delta-3}),\\
\mathrm{Ric}_{\tilde{g}}(\nabla\Phi,\tilde{\nu})r^{n-1}&=O_2(r^{n-2\delta-4}).
\end{align*}
Since $\sigma > \frac{n}{2}-2$, then, we can choose $\delta$ such that
\begin{align*}
|\mathrm{Ric}_{\tilde{g}}(\nabla\Phi,\tilde{\nu})|r^{n-1} \rightarrow 0 \hbox{ as } r\rightarrow +\infty.
\end{align*}
%$\sigma+\delta+2-\frac{n}{2}=2\delta$, and we find 
%\begin{align*}
%|\mathrm{Ric}_{\tilde{g}}(\nabla\Phi,\tilde{\nu})|r^{n-1}&\leq r^{-2\delta}.
%\end{align*}
%Since $\delta>0$, the above expression also goes to zero as we move to infinity. Then, the second term in the right-hand side of (\ref{Rigidity.1}) goes to zero as $r\rightarrow\infty$.
\medskip
\noindent \textbf{Estimates on $\Delta_{\tilde{g}}\Phi$}

Let us rewrite
\begin{align*}
\tilde{\nabla}_i\Phi&=-\frac{n-4}{n-2}u^{-2\frac{n-3}{n-2}}\tilde{\nabla}_iu,\\
\tilde{\nabla}_j\tilde{\nabla}_i\Phi&=-\frac{n-4}{n-2}\left( u^{-2\frac{n-3}{n-2}}\tilde{\nabla}_j\tilde{\nabla}_iu - 2\frac{n-3}{n-2} u^{-\frac{3n-8}{n-2}}\tilde{\nabla}_ju\tilde{\nabla}_iu \right),
\end{align*}
from which we get
\begin{align*}
\Delta_{\tilde{g}}\Phi&= - \frac{n-4}{n-2}u^{-2\frac{n-3}{n-2}}\Delta_{\tilde{g}}u + 2\frac{(n-4)(n-3)}{(n-2)^2} u^{-\frac{3n-8}{n-2}}|\tilde{\nabla} u|^2_{\tilde{g}}.
\end{align*}
Now, we can compute that
\begin{align*}
\Delta_{\tilde{g}}u&=\tilde{g}^{ij}\tilde{\nabla}_j\nabla_iu=\tilde{g}^{ij}\left( \nabla_j\nabla_iu - (\Gamma^{k}_{ij}(\tilde{g}) - \Gamma^{k}_{ij}(g))\nabla_ku \right),\\
&=\Delta_{g}u + (\tilde{g}^{ij}-{g}^{ij})\nabla_j\nabla^{}_iu - \tilde{g}^{ij}(\Gamma^{k}_{ij}(\tilde{g}) - \Gamma^{k}_{ij}(g))\nabla_ku
\end{align*}
and we can estimate $\tilde{g}^{ij}-g^{ij}=O_4(r^{-\delta})-O_4(r^{-\tau})=O_4(r^{-\delta})$ and $\Gamma^{k}_{ij}(\tilde{g}) - \Gamma^{k}_{ij}(g)=O_3(r^{-\delta-1})-O_3(r^{-\tau-1})=O_3(r^{-\delta-1})$, since $\sigma=\min\{\tau,n-2\}$, and also $\nabla u=O_3(r^{-\delta-1})$. Therefore
\begin{align}\label{Rigidity.2}
\begin{split}
\tilde{g}^{ij}(\Gamma^{k}_{ij}(\tilde{g}) - \Gamma^{k}_{ij}(g))\nabla_ku&=O_3(r^{-2\delta-2}),\\
(\tilde{g}^{ij}-{g}^{ij})\nabla_j\nabla_iu&=O_2(r^{-2\delta-2}).
\end{split}
\end{align}
On the other hand, from (\ref{scalcurvdeformation}) we find that\footnote{Notice that, actually, $R_{g}=O_2(r^{-\tau-2})$ which is stronger than they the decay we are using, but for our purposes keeping just one weight parameter given by $\delta$ is enough to achieve the required estimates.}
\begin{align}\label{Rigidity.3}
\Delta_{g}u=c_nR_{g} + c_n\underbrace{R_{g}}_{O_2(r^{-\delta-2})}\underbrace{\phi}_{O_4(r^{-\delta})}=c_nR_{g} + O_2(r^{-2\delta-2}).
\end{align}
Putting together (\ref{Rigidity.2})-(\ref{Rigidity.3}), we finally find that
\begin{align}
\Delta_{\tilde{g}}u&=c_nR_{g} + O_2(r^{-2\delta-2}),
\end{align}
which, in turn, implies
\begin{align}
\begin{split}
\Delta_{\tilde{g}}\Phi&= - \frac{n-4}{4(n-1)}u^{-2\frac{n-3}{n-2}}R_{g} + O_2(r^{-2\delta-2}),\\
\nabla(\Delta_{\tilde{g}}\Phi)&=- \frac{n-4}{4(n-1)}u^{-2\frac{n-3}{n-2}}\nabla R_{g} + O_1(r^{-2\delta-3}).
\end{split}
\end{align}
Therefore
\begin{align*}
\tilde{g}(\tilde{\nabla}\Delta_{\tilde{g}}\Phi,\tilde{\nu})&=\tilde{g}(\tilde{\nabla}\Delta_{\tilde{g}}\Phi,\nu^{e}) + \tilde{g}(\tilde{\nabla}\Delta_{\tilde{g}}\Phi,\tilde{\nu}-\nu^e),\\
&=\nabla_i\Delta_{\tilde{g}}\Phi \frac{x^i}{r} + \underbrace{(\tilde{g}_{ij}-\delta_{ij})\nabla^i(\Delta_{\tilde{g}}\Phi)}_{O_1(r^{-2\delta-3})}\frac{x^j}{r} + \underbrace{\tilde{g}(\tilde{\nabla}\Delta_{\tilde{g}}\Phi,\tilde{\nu}-\nu^e)}_{O_1(r^{-2\delta-3})},\\
&=- \frac{n-4}{4(n-1)}u^{-2\frac{n-3}{n-2}}\nabla_i R_{g}\frac{x^i}{r} + O_1(r^{-2\delta-3}),
\end{align*}
implying
\begin{align*}
\tilde{g}(\tilde{\nabla}\Delta_{\tilde{g}}\Phi,\tilde{\nu})r^{n-1}&=- \frac{n-4}{4(n-1)}\underbrace{u^{-2\frac{n-3}{n-2}}}_{\xrightarrow[r\rightarrow\infty]{}1}\partial_r R_{g}r^{n-1} + \underbrace{O_1(r^{n-2\delta-4})}_{\xrightarrow[r\rightarrow\infty]{}0},
\end{align*}
where we have used, as above, that $\delta>\frac{n}{2}-2$, implying we can choose $\delta>0$ such that $n-2\delta-4<0$. From all this and $d\omega_{\tilde{g}}=(1+O(r^{-\delta}))d\omega_r$, we find that
\begin{align*}
\begin{split}
\lim_{r\rightarrow\infty}\int_{S^{n-1}_r}\tilde{g}(\tilde{\nabla}\Delta_{\tilde{g}}\Phi,\tilde{\nu})d\omega_{r}&=- \frac{n-4}{4(n-1)}\lim_{r\rightarrow\infty}\left(\int_{S^{n-1}_1}\partial_r R_{g}r^{n-1}d\omega_{n-1} + \int_{S^{n-1}_1}\underbrace{\partial_r R_{g}O(r^{-\delta})r^{n-1}}_{O(r^{n-2\delta-4})}d\omega_{n-1}\right),\\
&=\frac{n-4}{4(n-1)}\mathcal{E}(g),
\end{split}
\end{align*}
Putting together the above analysis, from (\ref{Rigidity.1}), we find that 
\begin{align}\label{YamabePEthm.1}
\int_{M}\frac{n-4}{2}\Phi^{\frac{n+4}{n-4}}Q_{g}+ \frac{(n-4)}{(n-2)^2}|\mathrm{Ric}_{\tilde{g}}|^2_{\tilde{g}}\Phi \, dv_{\tilde{g}} =\frac{n-4}{4(n-1)}\mathcal{E}(g).
\end{align}
Thus, if $\mathcal{E}(g)=0$, then $Q_g \equiv 0$ and $\mathrm{Ric}_{\tilde{g}}\equiv 0$, which implies, through results such as those exposed in the proof of Proposition \ref{Maxprinciple.1}  that $\tilde{g}$ is flat and $M\cong \mathbb{R}^n$. Therefore, since $\tilde{g}=u^{\frac{4}{n-2}}g$, then $g$ is conformally-flat.

\medskip
\noindent\underline{II) \texorpdfstring{$n= 4$}{TEXT}}

%Once more, let us assume that $g$ satisfies $Q_{g},R_g\geq 0$; the decaying condition stated for positivity and assume that $\mathcal{E}(g)=0$. Also, in parallel to the above procedure, we start with a conformal deformation $\tilde{g}=u^{2}g$ to zero scalar curvature. That is, $u$ is the unique positive solution to
%\begin{align}
%\begin{split}
%\Delta_gu-R_gu=0,\\
%u\xrightarrow[r\rightarrow\infty]{} 1
%\end{split}
%\end{align}
%Fixing $\varphi=u-1$, the a above equation is equivalent to $\Delta_g\varphi - R_g\varphi=R_g\in H^s_{-\delta-2}$ for any $\delta\in (0,\tau)$, where $\tau>0$. Since $\Delta_g:H^s_{\sigma}\mapsto H^{s-2}_{\sigma-2}$ is an isomorphism for any $s>2$ and $\sigma\in (-2,0)$, then $\varphi\in H^s_{-\delta}$ for some $\tau>\delta> 0$.

\medskip
Now, fix $\Phi=-\ln(u)$ and rewrite $\tilde{g}=e^{-2\Phi}g$. Then, using the transformation rule (\ref{Qcurvtrans4d.1}) we find
\begin{align}
0\leq Q_{g}=e^{-4\Phi}(P_{\tilde{g}}\Phi + Q_{\tilde{g}}).
\end{align}
Since $R_{\tilde{g}}=0$, then $Q_{\tilde{g}}=-\frac{1}{2}|\mathrm{Ric}_{\tilde{g}}|^2_{\tilde{g}}$, which implies
\begin{align}\label{4drigidity.1}
0\leq \int_{D_r}\left(Q_{g}e^{4\Phi} + \frac{1}{2}|\mathrm{Ric}_{\tilde{g}}|^2_{\tilde{g}}\right)dV_{\tilde{g}}=\int_{S_r}\tilde{g}(\tilde{\nabla}\Delta_{\tilde{g}}\Phi,\tilde{\nu})d\omega_{\tilde{g}}+2\int_{S_r}\mathrm{Ric}_{\tilde{g}}(\nabla\Phi,\tilde{\nu})d\omega_{\tilde{g}}
\end{align}

\medskip
\noindent\textbf{Estimates on $\mathrm{Ric}_{\tilde{g}}(\nabla\Phi,\tilde{\nu})$}

We clearly have $u=1+O_4(r^{-\delta})$ for some $\delta< \sigma$. Then, from $\nabla\Phi=-\frac{1}{u}\nabla u=O_3(r^{-\delta-1})$, we find that
\begin{align}
\begin{split}
\mathrm{Ric}_{\tilde{g}}(\nabla\Phi,\tilde{\nu})r^{3}&=\underbrace{\mathrm{Ric}_{\tilde{g}}(\nabla\Phi,\nu_e)}_{O_3(r^{-2\delta-3})}r^3 + \underbrace{\mathrm{Ric}_{\tilde{g}}(\nabla\Phi,\tilde{\nu}-\nu_e)}_{O_3(r^{-3\delta-3})}r^{3}=O_3(r^{-2\delta}).
\end{split}
\end{align}
Combining this with $d\omega_{\tilde{g}}=(1+O(r^{-\delta}))d\omega_r$ we get
\begin{align}\label{4Destimate.1}
\int_{S_r}\mathrm{Ric}_{\tilde{g}}(\nabla\Phi,\tilde{\nu})d\omega_{\tilde{g}}=o(1).
\end{align}

\medskip
\noindent\textbf{Estimates on $\tilde{g}(\nabla\Delta_{\tilde{g}}\Phi,\tilde{\nu})$}

Straightforwardly we see that
\begin{align}
\begin{split}
\tilde{\nabla}_j\tilde{\nabla}_i\Phi&=-u^{-1}\tilde{\nabla}_j\tilde{\nabla}_iu + u^{-2}\tilde{\nabla}_j u \tilde{\nabla}_i u=-u^{-1}\tilde{\nabla}_j\tilde{\nabla}_iu + O_3(r^{-2\delta - 2}),\\
\tilde{\nabla}_j\tilde{\nabla}_iu&=\nabla_j\nabla_iu - \Delta\Gamma^k_{ji}\nabla_ku=\nabla_j\nabla_iu + O_3(r^{-2\delta-2}),
\end{split}
\end{align}
where $\Delta\Gamma^k_{ji}=\Gamma^{k}_{ij}(\tilde{g})-\Gamma^{k}_{ij}(g)$. Thus, we see that
\begin{align*}
\Delta_{\tilde{g}}\Phi&=-u^{-1}\Delta_gu - u^{-1}\underbrace{(\tilde{g}^{ij}-g^{ij})\nabla_j\nabla_iu}_{O_{3}(r^{-2\delta-2})} + O_3(r^{-2\delta - 2})=-u^{-1}\Delta_gu + O_3(r^{-2\delta-2}),\\
&=-c_nR_g + O_3(r^{-2\delta-2}),
\end{align*}
thus,
\begin{align*}
\nabla\Delta_{\tilde{g}}\Phi&=-c_n\nabla R_g + O_2(r^{-2\delta-3}),\\
\tilde{g}(\nabla\Delta_{\tilde{g}}\Phi,\tilde{\nu})&=-c_n\tilde{g}(\nabla R_g,\nu_e) - c_n\underbrace{\tilde{g}(\nabla R_g,\tilde{\nu}-\nu_e)}_{O_1(r^{-2\delta-3})} + O_2(r^{-2\delta-3}),
\end{align*}
implying
\begin{align*}
\int_{S_r}\tilde{g}(\nabla\Delta_{\tilde{g}}\Phi,\tilde{\nu})d\omega_{\tilde{g}}&=-c_n\int_{S_r}\tilde{g}(\nabla R_g,\nu_e)d\omega_{\tilde{g}} + O(r^{-2\delta}),\\
&=-c_n\int_{S_r}\tilde{g}(\nabla R_g,\nu_e)d\omega_{r} - \int_{S_r}\underbrace{\tilde{g}(\nabla R_g,\nu_e)}_{O(r^{-\tau-3})}O(r^{-\delta})d\omega_{r}  + O(r^{-2\delta}),\\
&=-c_n\int_{S_r}\tilde{g}(\nabla R_g,\nu_e)d\omega_{r} + o(1)\\
&=-c_n\int_{S_r} g(\nabla R_g,\nu_e)d\omega_{r} -c_n\int_{S_r} (\underbrace{u^2-1}_{o(1)})\underbrace{g(\nabla R_g,\nu_e)}_{O(r^{-\delta-3)}}d\omega_{r} + o(1)\\
&=-c_n\int_{S_r} g(\nabla R_g,\nu_e)d\omega_{r} + o(1).
\end{align*}
Therefore, we find that
\begin{align}\label{4destimate.2}
\lim_{r\rightarrow\infty}\int_{S_r}\tilde{g}(\nabla\Delta_{\tilde{g}}\Phi,\tilde{\nu})d\omega_{\tilde{g}}&=c_n\mathcal{E}(g).
\end{align}
Finally, putting together (\ref{4drigidity.1}),(\ref{4Destimate.1}) and (\ref{4destimate.2}), we find that
\begin{align}\label{YamabePEthm.2}
0\leq \int_{M}\left(Q_{g}e^{4\Phi} + \frac{1}{2}|\mathrm{Ric}_{\tilde{g}}|^2_{\tilde{g}}\right)dV_{\tilde{g}}=c_n\mathcal{E}(g),
\end{align}
which implies that if $\mathcal{E}(g)=0$, then $\mathrm{Ric}_{\tilde{g}}\equiv 0$ and therefore $g$ is conformally-flat with $Q_g\equiv 0$.

\medskip
Finally, $g= v^\frac{4}{n-4} \delta$ (or $g=e^{2v}\delta$ if $n=4$), with $\Delta^2 v=0$ and $\displaystyle \lim_{r\rightarrow\infty} v=1$ (or $\displaystyle \lim_{r\rightarrow\infty} v=0$), hence, by maximum principle, we get that $v\equiv 1$ (or $v\equiv 0$) which achieves the proof of the theorem.
\end{proof}

It is worth noticing that, looking more carefully to the above proof, we can weaken the hypotheses of the previous theorem and still get interesting results. Indeed, in the positivity part, we only use the fact that the scalar curvature is positive at infinity, while for the rigidity statement, we only use the fact that we can make the scalar curvature flat via a conformal transformation, which is ensured by assuming that $Y([g])> 0$ by theorem 5.1 of \cite{Maxwell}, and that $Q_g$ is non-negative. Furthermore, under this last condition, equations (\ref{YamabePEthm.1}) and (\ref{YamabePEthm.2}) also provide proofs of positivity. All these observations give us the following results.

\begin{thm}\label{PEthm2}
Let $(M^n,g)$ be an $n$-dimensional AE manifold, with $n\geq 3$, which satisfies the decaying conditions i) and ii) of Proposition \ref{Static.1} and such that $R_g> 0$ on $M\setminus K$, for some compact set $K$ then the fourth order energy $\mathcal{E}(g)\geq 0$.
\end{thm}

\begin{thm}\label{PEthm3}
Let $(M^n,g)$ be an $n$-dimensional AE manifold, with $n\geq 3$, which satisfies the decaying conditions i) and ii) of Proposition \ref{Static.1} and such that $Q_g\geq 0$ and $Y([g])> 0$, then  $\mathcal{E}(g)\geq 0$ with equality holding if and only if $(M,g)$ is isometric to $(\mathbb{R}^n,\cdot)$.
\end{thm}

Let us now notice that the $Q$-curvature assumption in Theorem \ref{PEthm3} cannot be weakened while keeping the rest of the hypotheses. This can be seen as follows: Consider $(\mathbb{R}^n,\delta)$, with $n\geq 3$, and let $\{h_k\}_{k=1}^{\infty}\in C^{\infty}(U)$, be a sequence of smooth compactly supported symmetric second rank tensor fields, $\bar{U}\subset\subset M$, with $||h_k||<\epsilon$, for some fixed small $\epsilon>0$, such that $h_k\xrightarrow[k\rightarrow\infty]{} 0$. Then, consider the sequence of metrics $g_k=\delta+h_k$, where the construction of the $h_k$ can be made so that $g_k$ are not conformally flat. Furthermore, this construction can be fit so that $\{g_k\}_{k=1}^{\infty}$ are all Yamabe positive.\footnote{For an explicit construction of this see Example 1 in \cite{Avalos-Lira}} Therefore, there are conformal factors $\{u_k=1+\varphi_k\}$, with $\varphi_k\in H^s_{-\delta}$, with $\delta<n-2$ arbitrary, such that $\Delta_{g_k}u_k-R_{g_k}u_k=0$.\footnote{In fact, from Theorem 1.17 in \cite{Bartnik}, it follows that
\begin{align*}
u_{k}=1+\frac{a_k}{r^{n-2}} + o(r^{-(n-2)}),
\end{align*}
where $a_k$ are constants.} This means that $R_{u_k^{\frac{4}{n-2}}g_k}=0$ and thus $Q_{u_k^{\frac{4}{n-2}}g_k}=-c_2(n)|\mathrm{Ric}_{u_k^{\frac{4}{n-2}}g_k}|^2_{u_k^{\frac{4}{n-2}}g_k}<0$. Furthermore, all the metrics $u_k^{\frac{4}{n-2}}g_k$ satisfy the decaying conditions of Proposition \ref{Static.1} and are, by construction, Yamabe positive. Nevertheless, from their asymptotics $u_k^{\frac{4}{n-2}}{g_k}_{ij}=(1+\frac{a_{k,n}}{r^{n-2}})\delta_{ij} + o_k(r^{-(n-2)})$ we see that $\mathcal{E}(g_k)=0$ for all $k$, although none of the $g_k$ are conformally flat. Finally, by considering $k$ sufficiently large, we can make $-\frac{1}{k}<Q_{u^{\frac{4}{n-2}}_kg_k}<0$, which shows that there are Yamabe positive AE metrics satisfying conditions (i)-(ii) of Proposition \ref{Static.1} which have negative $Q$-curvature which is as small as we want, zero energy and are not conformally flat.

%$||\varphi_k||_{H^s_{-\delta}}\leq C||R_{g_k}||_{H^{s-2}_{-\delta-2}}$

\bigskip

%All of the above would establish the following theorem:
%\begin{thm}\label{PEthm}
%Let $(M^n,g)$ be an AE-Riemannian manifold with $n\geq 3$ satisfying the decaying condition $\tau>\max\left(\frac{n}{2}-2,0\right)$ and $Q_g\in L^1(M,dV_g)$. If $R_g,Q_g\geq 0$, then $\mathcal{E}(g)\geq 0$ and equality holds iff $M\cong \mathbb{R}^n$ and $g$ is conformally flat. 
%\end{thm}

%\begin{remark}
In dimensions $n=3,4$, under the above hypotheses, we fall into a curious situation, since any such $n$-dimensional AE-manifold has $\mathcal{E}(g)=0$ as a consequence of the fall-off conditions, since  $ \nabla R_g= 0(r^{-\tau-3})$ decreases faster than the volume of the $S_r^{n-1}$.

% Thus, we must conclude that the set of AE-manifolds\footnote{In dimension 4 the fall of condition is simply $\tau>0$!} with $Q_g\in L^1(M,dV_g)$ satisfying $R_g,Q_g\geq 0$ is fully characterised: they are conformally flat with $M\cong \mathbb{R}^4$. We will come back to this in the next section.
%\end{remark}

\begin{coro}\label{PEthm4d}
Any n-dimensional AE-Riemannian manifold $(M^n,g)$, with $n\in\{3,4\}$, such that $Q_g\geq 0$ and $Y([g])> 0$, is isometric to $(\R^n,\delta)$.
\end{coro}

\subsection*{Rigidity in critical cases}

The aim of this section is to comment further on the especial case that occurs in dimension four. In particular, the fact that the energy $\mathcal{E}(g)$ is always zero has some topological consequences as a corollary. It is worth to contrast this with the second order case associated to the ADM energy. In this last case, the critical case is given in dimension two, where asymptotic flatness is too strong a condition in order to detect any meaningful information concerning the ADM energy (for details, see Chapter 3 in \cite{Lee}). Let us start by briefly commenting on this case as a warm-up.

%We start this section with a proof of the classical PMT in dimension $2$. 
Let $(M,g)$ be an asymptotically flat surface with $\tau>0$ and non-negative integrable Gaussian curvature. Let us consider a chart a infinity, and recall the definition of the ADM energy
$$m(g)=\lim_{R\rightarrow +\infty} \int_{\partial B_R}
(\partial_ig_{ij}-\partial_jg_{ii})\nu^j \, ds,$$
where $\nu$ is the outer normal. Because of the decreasing assumptions, we easily see that the mass must vanish. Let us write the Gauss-Bonnet formula which gives
$$2\chi(M)=\int_{M\setminus B(0,R)^c} K \, d\sigma + \int_{\partial B(0,R)} k_g \, ds .$$
Here we write $\chi(M)$ instead of $\chi(M\setminus B(0,R)^c)$ since they are equal for $R$ large. Then in polar coordinate we have,%\footnote{In the Gauss-Bonnet theorem we use the \textit{signed} geodesic curvature, defined by $k_g\doteq g(\nabla_{\gamma'}\gamma',\nu)$, where $\gamma'$ is parametrized by arclength and $\{\gamma',\nu\}$ is a positively oriented orthonormal basis along $\gamma$. Thus, in our case $\gamma'=\frac{1}{r}\frac{\partial}{\partial \theta}$ and $\nu=-\frac{\partial}{\partial r}$.}
\begin{align*}
k_g&=-\frac{1}{r^2}g\left(\nabla_{\partial_\theta} \partial_\theta,\partial_r\right)=-\frac{1}{r^2}\Gamma^{r}_{\theta\theta}(g) + O\left(\frac{1}{r^{1+\tau}}\right)=-\frac{1}{r^2}\Gamma^{r}_{\theta\theta}(\delta) + O\left(\frac{1}{r^{1+\tau}}\right) \\
&=\frac{1}{r} + O\left(\frac{1}{r^{1+\tau}}\right),
\end{align*}
hence passing to the limit we obtain that 
$$2\pi \chi(M)-2\pi =\int_M K\, d\sigma.$$
But we know that 
$$\chi(M)=2-2g-r,$$
where $g$ is the genus of $M$ and $r$ its number of ends. Then we necessary get that $g=0$, $r=1$ and $K\equiv0$. Hence we recover the following well-known proposition, see chapter 3 \cite{Lee},
\begin{prop}
    Let $(M,g)$ an asymptotically flat surface with non-negative Gauss curvature. Then, $(M,g)$ is the Euclidean plane.
\end{prop}

Let us turn to the $4$-dimensional case. Let $(M,g)$ be a 4-dimensional AE manifold with $\tau >0$ with only one end such that $\vert W_g\vert^2$ is integrable. From Corollary \ref{PEthm4d}, we know that the mass is necessarily vanishing
\begin{align}
    \mathcal{E}(g)=-\lim_{R\rightarrow\infty}\int_{S_{R}}\partial_rR_g r^{3}d\omega_{3}=0.
\end{align}

Along the lines of the $2$-dimensional case, let us try to get a similar proof of the rigidity result. First, we apply  the Gauss-Bonnet-Chern formula with boundary (see \cite{Chen} for instance)
$$32\pi^2 \chi(M) =\int_{K_R} \vert W\vert^2 +16\sigma_2(S_g)\, dv + 8\int_{\partial B(0,R)} B_g \, dv$$
where $K_R=\left(M\setminus \overline{B(0,R)}\right)^c$, $S_g=\frac{1}{2} \left( \mathrm{Ric} -\frac{1}{6} R_g g\right)$ is the Schouten tensor and $B_g=\frac{1}{2} R_gH-\mathrm{Ric}(\nu,\nu)H-R_{\gamma \alpha \beta \gamma } \mathrm{II}^{\alpha\beta} +\frac{1}{3} H^3 -H\vert \mathrm{II}\vert^2 +\frac{2}{3} tr(\mathrm{II}^3)$, where $\mathrm{II}$ and $H$ are respectively the second form and the mean curvature of a hypersurface $\Sigma\hookrightarrow M$, $H=\mathrm{tr}_{\Sigma}\mathrm{II}$, taken with respect to the inner unit normal to $\Sigma$, the Greek indices are indices tangent to $\Sigma$, and $\mathrm{II}^3$ is defined in local coordinates by ${\mathrm{II}^3}^i_j=\mathrm{II}^i_k\mathrm{II}^k_l\mathrm{II}^l_j$.
%\footnote{$h=\sigma_1(L)$} of $\partial B(0,R)$.\\
Moreover, in dimension $4$, we have
$$Q_g =-\frac{1}{6} \Delta_g R_g +4 \sigma_2(S_g),$$
and thanks to our decreasing assumption, we get 
$$B_g= \frac{2}{r^3} +O\left( \frac{1}{r^{3+\tau}}\right)$$
which gives, using once more our decreasing assumption, that

\begin{align}
\begin{split}
    32\pi^2 (\chi(M)-1)&= \int_{K_R} \vert W\vert^2 + 4 Q_g +\frac{2}{3} \Delta_g R_g \, dv_g +O\left( \frac{1}{R^\tau}\right)\\
    &= \int_{K_R} \vert W\vert^2 +4 Q_g  \, dv_g +O\left( \frac{1}{R^\tau}\right)
\end{split}
\end{align}
Hence passing to the limit, we get

$$32\pi^2 (\chi(M)-1)= \int_{M} \vert W\vert^2 + 4 Q_g  \, dv_g $$
Let us set 
$$\kappa_g=\int_M Q_g \, dv_g  .$$
One can use the above expression to recover Corollary \ref{PEthm4d} in the four dimensional case as consequence of the conformal invariance of $\kappa_g$. Also, we can easily deduce the following curvature-topology proposition. This proposition is not really new but we just would like to put them in perspective with our rigidity result

\begin{prop}\label{propc}
    Let $(M,g)$ be an AE 4-manifold with $\kappa_g \geq0$, then $\chi(M) \geq 1$ with equality if $(M,g)$ is locally conformally flat with $\kappa_g= 0$.\\
    In particular, if $\kappa_g \geq 0$ and $Y([g])> 0$ or the second betti number vanishes, then $\kappa_g= 0$ and $(M,g)$ is conformal to the euclidean $\R^4$.
\end{prop}

\subsection{The Q-curvature positive mass theorem}

In this section we will make contact with a series of recent results associated with the positive mass theorem for the Paneitz operator, namely \cite{Malchiodi1,hangyang2,Raulot}. In particular, we will show that these results follow from Theorem \ref{PEthm3}. %Let us highlight that these positive mass theorems can be seen to play a crucial role in the prescription of constant positive $Q$-curvature in a conformal class, very much in the same way the positive mass theorem associated to the 

Let us consider a closed manifold $(M^n,g)$ with $n\geq 5$. If, the Panietz operator is coercive, that is to say 
\[ \inf_{u\in H^2(M)\, ;\, \Vert u\Vert_2=1} \int_M P_g(u)u \, dv_g >0,\]
then it admits a Green functions $G_{P_g}$. This is in particular guaranteed if $R_g \geq 0$ and  $Q_g$ is semi-positive, see proposition B of \cite{Malchiodi1}, or more generally by the fact that 
\[ Y_4([g])= \inf_{u\in H^2(M)\, ;\, \Vert u\Vert_{2^\#}=1} \int_M P_g(u)u \, dv_g >0,\]
where $2^\#=\frac{2n}{n-4}$. This last infimum has the advantage to be conformally invariant and it plays a similar role to the Yamabe invariant for the $Q$-curvature. Through this section, we will assume that $Y_4([g])>0$ and therefore $G_{P_g}$ exists for every $g\in [g]$.

Nevertheless, contrary to the conformal Laplacian, nothing here guarantees that $G_{P_{g}}$ is positive. This will be one of our assumptions, which is in particular satisfied if $Y([g])\geq 0$ and $Q_{\tilde{g}}$ is semi-positive for some conformal metric $\tilde{g}$, due to lemma 3.2 in \cite{hangyang1}.
 
\begin{remark}
\label{remG} 
In fact Hang and Yang assume $Y([g])>0$, but if $Y([g])=0$ and there is some $\tilde{g}\in [g]$ with $Q_{\tilde{g}}$ semi-positive, then there exists $\bar{g}=u^{\frac{4}{n-2}}\tilde{g}$ such that $R_{\bar{g}}=0$ which implies that $Q_{\bar{g}}\leq 0$ which contradicts the fact that%\footnote{Under these conditions, fix $\tilde{g}$, which is also Yamabe zero, and thus there is some $u>0$ such that $R_{\bar{g}}=0$ for $\bar{g}=u^{\frac{4}{n-2}}\tilde{g}$. Then, defining $v\doteq u^{\frac{n-4}{n-2}}$, we have that $\bar{g}=v^{\frac{4}{n-4}}\tilde{g}$, and we can appeal to the transformation rule for the Paneitz operator to get
%\begin{align*}
%Q_{\bar{g}}=\frac{2}{n-4}v^{-\frac{n+4}{n-4}}\left(\Delta^2_{\tilde{g}}v + \mathrm{div}_{\tilde{g}}\left(\left( 4S_{\tilde{g}} - (n-2)\sigma_1(S_{\tilde{g}})\tilde{g} \right)(\nabla v,\cdot)\right) + \frac{n-4}{2}Q_{\tilde{g}}v\right).
%\end{align*}}
\[\int_M Q_{\tilde{g}} v\, dv_{\tilde{g}} =\int_M v^\frac{n+4}{n-4} Q_{\bar{g}} \, dv_{\tilde{g}} ,\]
with $v =u^{\frac{n-4}{n-2}}$. Hence the existence of a conformal metric with semi-positive $Q$-curvature forces the Yamabe invariant to be positive.

In fact, if $Y([g])\geq 0$ , thanks to, theorem 1.1 of \cite{hangyang1}, the existence of a positive $G_{P_g}$ is equivalent to the existence of a conformal metric $\tilde{g}$ with semi-positive $Q$-curvature.
\end{remark}

Let now us focus on the expansion of the Green function around a singularity. If we assume that $5\leq n\leq 7$ or $g$ is locally conformally flat, then, in conformal normal coordinates $\{x^i\}$ for the conformal metric $\tilde{g}$, the Green function $G_{P}$ of $P_{\tilde{g}}$ admits an expansion of the form (see Proposition 2.5 in \cite{Malchiodi1})
\begin{align}\label{Greenfunc.1}
G_P(p,x)=\frac{\gamma_n}{r^{n-4}}+\alpha+O_4(r),
\end{align} 
where $r(x)\doteq d_{\tilde{g}}(p,x)$, $\gamma_n\doteq \frac{1}{2(n-2)(n-4)\omega_{n-1}}$ and $\alpha$ is a constant \textbf{called the mass}.
\begin{remark}
\label{remsG}
As remarked in \cite{hangyang1}, under the condition $\mathrm{Ker}(P_g)=0$ (which is itself conformally invariant), the sign of the Green function is a conformal invariant since,
\[G_{P_{u^\frac{4}{n-4}g}}(p,q)=u(p)^{-1}u(q)^{-1} G_{P_g}(p,q).\]
 Hence, the sign of the mass is also a conformal invariant. %This will be key in our analysis.
\end{remark}

This terminology arises in analogy to the Yamabe problem, where the Green function $G_{L_g}$ of the conformal Laplacian is involved and, in important cases, it admits a similar expansion to (\ref{Greenfunc.1}). In the case of the conformal Laplacian, it was an observation of R. Schoen that the constant which appears in the place of $\alpha$ is precisely the ADM mass of the AE-manifold obtained by via the stereographic projection $(M\backslash \{ p \},G^{\frac{4}{n-2}}_{L}g)$. This showed that the resolution of the positive mass conjecture in general relativity would amount to completing the resolution of the Yamabe problem, which concerned the cases $Y(g)>0$ and dimensions $n=3,4,5$ or $M$ locally conformally flat (see \cite{Schoen1} and \cite{Lee-Parker}).

The positive mass theorem presented in \cite{Malchiodi1} as Theorem 2.9 states that under the conditions described above which provide us with the expansion (\ref{Greenfunc.1}), the mass satisfies $\alpha\geq 0$ with equality holding if and only if $(M,g)$ is conformal to the round sphere. This came about as a generalisation of \cite{Raulot}, where this result was proven in the conformally flat case, and, in turn, all this was generalised in \cite{hangyang2}, where the condition $R_g\geq 0$ was replaced by $Y([g])>0$. Furthermore, in Theorem 1.4 of \cite{hangyang1}, this positive mass theorem is used to solve the $Y_4([g])>0$ $Q$-curvature prescription problem in very much the same spirit as the usual positive mass theorem was used by Schoen in \cite{Schoen1}. We intend to show that all these results are special cases of the positive energy theorem presented above. 

Finally, let us notice that, for instance, in the $Q$-curvature prescription problem it is the sign of the mass that is actually important rather than its precise value. This becomes especially useful when it is combined with the additional observation that the sign of the mass is itself a conformal invariant, which was highligted in Remark \ref{remsG}. Thus, being concerned with a conformal problem, given $(M^n,[g])$, we will always consider a choice of $g$ satisfying the conformal normal coordinate properties (\ref{confcoord.0}),(\ref{confcoord.1}),(\ref{confcoord.2}). Notice that from the discussion presented in that appendix, this can always be achieved by first going to a related conformal metric. %Then, if we establish that $\alpha\geq 0$ for some preferred element in the conformal class, the same follows for every element.

In order to prepare for the main statements of this section, let us first establish a couple of preliminary results.

\begin{prop}\label{blowup.0}
Let $(M^n,g)$ be a closed manifold satisfying $n\geq 5$ and whose Panietz operator admits a positive Green function $G_P$ with an expansion as (\ref{Greenfunc.1}) around some point $p\in M$. Then, the manifold $(\hat{M}\doteq M\backslash\{p\},\hat{g}\doteq G_P(p,\cdot)^{\frac{4}{n-4}}g)$ is an asymptotically flat manifold of order $\tau=1$ if $n=5$ and $\tau=2$ if $n> 5$. Furthermore, either if $5\leq n\leq 7$ or $g$ is flat around $p$, then $\mathcal{E}(\hat{g})=8(n-1)(n-2)\omega_{n-1}\gamma_n\alpha$.
\end{prop}

\begin{proof}
Let us fix the conformal normal coordinates $\{x^i\}$ where (\ref{Greenfunc.1}) holds. Let us appeal to the conformal normal coordinate construction of order $N\geq 4$ described Appendix B so that (\ref{confcoord.1})-(\ref{confcoord.2}) hold, and start by considering the following expansion near $p$ 
\begin{align}\label{blowup.1}
g_{ij}=\delta_{ij}+\frac{1}{3}R_{iklj}(p)x^kx^l + \frac{1}{6} R_{iklj,a}(p)x^kx^lx^a + O(r^4),
\end{align}
Now, consider the inverted coordinates $z^i=\gamma_n^{\frac{2}{n-4}}\frac{x^i}{r^2}$ in a neighborhood of $p$ and define $\rho^2\doteq |z|^2$, so that $\rho^2=\gamma_n^{\frac{4}{n-4}}r^{-2}$ and $x^{i}=\gamma_n^{\frac{2}{n-4}}\frac{z^i}{\rho^2}$. Then, it holds that $\gamma_nr^{-(n-4)}=\gamma_n^{-1}\rho^{n-4}$ and
\begin{align}
\frac{\partial}{\partial z^i}=\gamma_n^{\frac{2}{n-4}}\rho^{-2}\left(\delta^j_i - 2\rho^{-2}z^iz^j \right)\frac{\partial}{\partial x^j},
\end{align}  
which, appealing to (\ref{Greenfunc.1}), implies
\begin{align*}
\hat{g}_{ij}(z)&=\left( \gamma^{-1}_n\rho^{n-4} + \alpha +O_4(\rho^{-1}) \right)^{\frac{4}{n-4}}g(\partial_{z^i},\partial_{z^j}),\\
&= \left( 1 + \frac{\gamma_n\alpha}{\rho^{n-4}} + O_4(\rho^{3-n}) \right)^{\frac{4}{n-4}}\left(\delta^i_k - 2\rho^{-2}z^iz^k \right)\left(\delta^j_l - 2\rho^{-2}z^jz^l \right)g_{kl}\left(\gamma_n^{\frac{2}{n-4}}\rho^{-2}z\right).
\end{align*}
Invoking (\ref{blowup.1}), this implies
\begin{align*}
\hat{g}_{ij}(z)&=\left( 1 + \frac{\gamma_n\alpha}{\rho^{n-4}} + O_4(\rho^{3-n}) \right)^{\frac{4}{n-4}}h_{ij}(z),
\end{align*}
where 
\begin{align*}
h_{ij}(z)&=\left( \delta_{k}^i\delta_{l}^j-2 \frac{z^iz^k}{\rho^2}\delta_l^j-2 \frac{z^jz^l}{\rho^2}\delta_k^i+4 \frac{z^iz^jz^kz^l}{\rho^4} \right)\left( \delta_{kl}+\frac{\gamma_n^{\frac{4}{n-4}}}{3}R_{kmnl}(p)\frac{z^m z^n}{\rho^4}\right.\\
&+\left.\frac{\gamma^{\frac{6}{n-4}}_n}{6}R_{kmnl,p}(p) \frac{z^m z^n z^p}{\rho^6} +O_4\left(\frac{1}{\rho^4}\right) \right),\\
&= \delta_{ij} + H^2_{ij}(z)  + H^3_{ij}(z) + O_4\left(\frac{1}{\rho^4}\right)
\end{align*}
where
\begin{align*}
\gamma^{-\frac{4}{n-4}}_nH^2_{ij}&=\frac{1}{3}R_{imnj}(p)\frac{z^m z^n}{\rho^4} - \frac{2}{3}\underbrace{R_{kmnj}(p)z^kz^m}_{=0}\frac{z^n z^i}{\rho^6} -\frac{2}{3}\underbrace{R_{imnl}(p)z^nz^l}_{=0}\frac{z^m z^j}{\rho^6}  \\
&+\frac{4}{3}\underbrace{R_{kmnl}(p)z^kz^m}_{=0}\frac{z^nz^l z^iz^j}{\rho^8},\\
&=\frac{1}{3}R_{imnj}(p)\frac{z^m z^n}{\rho^4}
\end{align*}
and
\begin{align*}
\gamma^{-\frac{6}{n-4}}_nH^3_{ij}&=\frac{1}{6}R_{imnj,a}(p)\frac{z^m z^nz^a}{\rho^6} -\frac{1}{3}\underbrace{R_{kmnj,a}(p)z^kz^m}_{=0}\frac{z^n z^iz^a}{\rho^8} - \frac{1}{3}\underbrace{R_{imnl,a}(p)z^nz^l}_{=0}\frac{z^mz^jz^a}{\rho^8} \\
&+ \frac{2}{3}\underbrace{R_{kmnl,a}(p)z^kz^m}_{=0}\frac{z^nz^l z^iz^jz^a}{\rho^{10}},\\
&=\frac{1}{6}R_{imnj,a}(p)\frac{z^m z^nz^a}{\rho^6}.
\end{align*}
Putting together all the above, we find that
\begin{align*}
\hat{g}_{ij}(z)&=\left( 1 + \frac{4}{n-4}\frac{\gamma_n\alpha}{\rho^{n-4}} \right)\delta_{ij} + \frac{\gamma^{\frac{4}{n-4}}_n}{3}R_{iabj}(p)\frac{z^az^b}{\rho^4} +  \frac{\gamma^{\frac{6}{n-4}}_n}{6}R_{iabj,c}(p)\frac{z^az^bz^c}{\rho^6} \\
&+ O_4(\rho^{-(n-3)}) +  O_4(\rho^{-4}),
\end{align*}
which establishes the asymptotic flatness condition claimed in the proposition. Notice that if $5\leq n\leq 7$, then, from Proposition \ref{Static.1}, we know that $\mathcal{E}(\hat{g})$ is well-defined and the above expansion implies
\begin{align*}
\partial_a\hat{g}_{ij}&= -4\gamma_n\alpha\frac{z^a}{\rho^{n-2}}\delta_{ij} + \partial_a H^2_{ij} +\partial_a H^3_{ij} +O_3\left(\rho^{-(n-2)}\right),\\
\partial_{ba}\hat{g}_{ij}&=-4\gamma_n\alpha\left(\frac{\delta_{ab}}{\rho^{n-2}}-(n-2)\frac{z^bz^a}{\rho^{n}}\right)\delta_{ij} + \partial_{ba} H^2_{ij} +\partial_{ba} H^3_{ij} +O_2\left(\rho^{-(n-1)}\right).
\end{align*}
%It is not difficult to explicitly compute that
%\begin{align*}
%\begin{split}
%\partial_{aa} H^2_{ii}&=\frac{\gamma^{\frac{4}{n-2}}_n}{3}\underbrace{R_{ab}(p)}_{=0}\Big( 4\left(2 - n  \right)\rho^{-2}z^bz^a + 2\delta_{ab} \Big)\rho^{-4}=0,\\
%\partial_{aa} H^3_{ii}&=\frac{\gamma^{\frac{6}{n-2}}_n}{6}\underbrace{R_{ab,c}(p)\underbrace{\Big( \left( n   - 23\right)\rho^{-8} z^az^bz^c + 2\rho^{-6}\big(\delta_{ab}z^c  + \delta_{ca}z^b + \delta_{bc}z^a \big)\Big)}_{\text{symmetric}}}_{=0}=0,
%\end{split}
%\end{align*}
%where we have appealed to (\ref{confcoord.1}). The above implies
Appealing to (\ref{confcoord.1})-(\ref{confcoord.2}), it is not difficult to explicitly compute that $\partial_{aa} H^2_{ii}$ and $\partial_{aa} H^3_{ii}$ both vanish. This implies
\begin{align}\label{blowup.2}
\partial_{aa}\hat{g}_{ii}&=-8n\gamma_n\alpha \rho^{-(n-2)} + O_2(\rho^{-(n-1)}).
\end{align}
Similarly, appealing to (\ref{confcoord.1})-(\ref{confcoord.2}), it also holds that $\partial_{ij} H^2_{ij}=0$ and $\partial_{ij} H^3_{ij}=0$, and therefore we find that
%\begin{align*}
%\partial_{ij} H^2_{ij}&=\frac{\gamma^{\frac{4}{n-4}}_n}{3}\Big(  4\rho^{-6}R_{ib}(p) z^iz^b + \rho^{-4}R(p)  \Big)=0,\\
%\partial_{ij} H^3_{ij}&=\frac{\gamma^{\frac{6}{n-4}}_n}{6}\Big( 6\rho^{-8}R_{ib,c}(p)z^iz^bz^c + \rho^{-6}\big(-R_{c}(p)z^c - 2R_{ib,i}(p)z^b  \big)\Big)=0,
%\end{align*}
%where we have now appealed to both (\ref{confcoord.1})-(\ref{confcoord.2}). Therefore, we find that
\begin{align}\label{blowup.3}
\partial_{ij}\hat{g}_{ij}&=- 8\gamma_n\alpha\rho^{-(n-2)} + O_2(\rho^{-(n-1)}).
\end{align}
Thus, putting together (\ref{blowup.2})-(\ref{blowup.3}), we see that
\begin{align}\label{blowup.4}
\left(\partial_{caa}\hat{g}_{ii}(z)-\partial_{cij}\hat{g}_{ij}\right)\frac{z^c}{\rho}&=8\gamma_n\alpha(n-1)(n-2) \rho^{-(n-1)} + O_1(\rho^{-n}),
\end{align}
implying that
%\begin{align*}
%\partial_{cba}\hat{g}_{ij}&= -\frac{4\alpha}{\gamma_n}\left(-(n-2)\frac{\delta_{ab}z^c}{\rho^{n}}-(n-2)\frac{\delta_{bc}z^a}{\rho^{n}}-(n-2)\frac{z^b\delta_{ac}}{\rho^{n}}+n(n-2)\frac{z^bz^az^c}{\rho^{n+2}}\right)\delta_{ij}\\ 
%&+ \partial_{cba} H^2_{ij} +\partial_{cba} H^3_{ij} +O_1\left(\frac{1}{\rho^7}\right).
%\end{align*}
%All this implies that
%\begin{align*}
%\left(\partial_{caa}\hat{g}_{ii}(z) - \partial_{cij}\hat{g}_{ij}(z) \right)=\frac{8\alpha (n-1)(n-2)}{\gamma_n}\frac{z^c}{\rho^{n}}, 
%\end{align*}
%Moreover the term $\left(\partial_{caa}H^2_{ii}(z) - \partial_{cij}H^2_{ij}(z) \right)\frac{z^c}{\rho}$  is a linear combination of product of even number of $z^i$ hence when we integrate it is going to trace the Riemann tensor, so the result will be proportional to $\mathrm{Ric}_g(p)$ which is zero in conformal normal coordinates. And $\left(\partial_{caa}H^2_{ii}(z) - \partial_{cij}H^2_{ij}(z) \right)\frac{z^c}{\rho}$ is a linear combination of product of odd number of $z^i$, hence the integral is necessary zero. Finally we have
\begin{align}\label{blowup.5}
\mathcal{E}(\hat{g})&=\lim_{\rho\rightarrow\infty}\int_{S_{\rho}}\left(\partial_{caa}\hat{g}_{ii}(z) - \partial_{cij}\hat{g}_{ij}(z) \right)\frac{z^c}{\rho}d\omega_{\rho}=\omega_{n-1} 8\gamma_n\alpha (n-1)(n-2).
\end{align}

Finally, in $g$ is flat near $p$, without restriction on the dimension, we know that the expansion (\ref{Greenfunc.1}) holds in rectangular coordinates around $p$, where $g_{ij}(x)=\delta_{ij}$. Thus, the same computations as above show that 
\begin{align*}
\hat{g}_{ij}(z)&=\left( 1 + \frac{4}{n-4}\frac{\gamma_n\alpha}{\rho^{n-4}} \right)\delta_{ij} + O_4(\rho^{-(n-3)}) .
\end{align*}
In this case the order of decay is improved for $n\geq 7$ and in particular we know that $\mathcal{E}(\hat{g})$ is well-defined. From the above expression, the same computations that led us to (\ref{blowup.5}) prove that $\mathcal{E}(\hat{g})=8 \omega_{n-1}(n-1)(n-2)\gamma_n\alpha $.
\end{proof}

%\bigskip
In order to apply the positive energy theorem to establish that $\alpha\geq 0$, we first need to check that $\hat{g}$ satisfies its hypotheses. With this in mind, consider the following proposition.

\begin{prop}\label{blowup-scalcurv2}
Consider a closed Riemannian manifold $(M^n,g)$, with $ n\geq 5$, which admits a conformal metric with positive $Q$-curvature such that $Y([g])\geq 0$. Then, there exists a conformal metric $\tilde{g}$ such that the asymptotically flat manifold $(\hat{M}=M\backslash\{p\},\hat{g}=G_{P_{\tilde{g}}}^{\frac{4}{n-4}}\tilde{g})$ satisfies $Y([\hat{g}])> 0$ and $Q_{\hat{g}}\equiv 0$.
\end{prop}

\begin{proof}
Thanks to Remark \ref{remG},  $Y([g]) >0$ and $G_{P_{\tilde{g}}}$ exists and is positive for every $\tilde{g}\in [g]$. Then, thanks to proposition \ref{blowup.0}, we only have to prove that $Y([\hat{g}]) > 0$. Since $Y([g])>0$, let $G_{L_{\tilde{g}}}$ be the Green's function of the conformal Laplacian. We trivially get  that $\bar{g}= \left(\frac{G^2_L}{G_P^\frac{2(n-2)}{n-4}} \right)^\frac{2}{n-2}\hat{g}$ is scalar flat.
\end{proof}

\medskip

Equipped with the above two propositions, we can recover the following theorem, originally proved by Hang-Yang in \cite{hangyang1}, and which is an extension of results of Gursky-Malchiodi \cite{Malchiodi1}
and Humbert-Raulot \cite{Raulot}.

\begin{thm}\label{QPM.2}
Let $(M,g)$ be a closed $n$-dimensional Riemannian manifold, with $ 5\leq n\leq 7$  or $n\geq 8$ and  locally conformally flat around some point $p\in M$. If $Y([g])\geq 0$ and $(M,g)$ admits a conformal metric with semi-positive $Q$-curvature, then the  mass of $G_{P}$ at $p$ is non-negative and vanishes if and only if $(M,g)$ is conformal to the standard sphere.
\end{thm}

It is important to note that here the hypothesis are conformally invariant. In dimension $n\geq 6$, thanks to Corollary 1.1 \cite{GHL} and remark \ref{remG}, they are equivalent to $Y([g])> 0$ and $Y_4^*([g]) >0$,
where 

$$Y_4^*([g]) = \frac{n-4}{2} \inf_{\tilde{g}\in [g], R_{\tilde{g}}>0} \frac{\int_M Q_{\tilde{g}} \, dv_{\tilde{g}}}{\mathrm{Vol}_{\tilde{g}} (M)}.$$
Of course $Y_4^*([g])$ is also conformally invariant.

\begin{proof}
Under these conditions, we know that the Green function $G_P$ exists and is positive for every element in $[g]$. From Proposition \ref{blowup.0} we also know that there is a conformal metric $\tilde{g}$ such that the manifold $(\hat{M}=M\backslash\{p\},\hat{g}=G_{P_{\tilde{g}}}^{\frac{4}{n-4}}\tilde{g})$ is AE, satisfies the decay assumptions of Theorem \ref{PEthm3} and, furthermore the energy is positively proportional to the mass of $G_{P_{\tilde{g}}}$. Finally, from Proposition \ref{blowup-scalcurv2}, we know that $Y([\hat{g}])>0$ and $Q_{\hat{g}}\equiv 0$. Therefore, $(\hat{M},\hat{g})$ satisfies all the hypotheses of Theorem \ref{PEthm3}, and thus the non-negativity follows directly from $\mathcal{E}(\hat{g})$ being positively proportional to $\alpha$. Finally, if $\alpha=0$, we find that $\hat{M}$ is isometric to $\mathbb{R}^n$. Being $M$ the one point compactification of $\hat{M}$, we see that $M\cong S^n$ and $g$ is conformal to the round metric.
%Let us now notice that putting together Propositions \ref{blowup.0} and \ref{blowup-scalcurv2}, we see that given a closed Riemannian manifold $(M^n,g)$ satisfying the hypotheses of Proposition \ref{blowup-scalcurv2}, we can construct the an associated $Q$-flat AE manifold $(\hat{M}=M\backslash\{p\},\hat{g}=G_{P_{\tilde{g}}})^{\frac{4}{n-4}}\tilde{g})$ with $\tilde{g}\in [g]$ and, furthermore, from Proposition \ref{blowup.0}, the energy $\mathcal{E}(\hat{g})$ is well-defined and positively proportional to the mass $\alpha$. 

\end{proof}
 
%\begin{proof}
%Considering $g_\lambda= u_\lambda^\frac{4}{n-4}g$, as in proof of proposition \ref{blowup-scalcurv}, we get that $g_\lambda$ is an AE euclidean metric with mass positively proportional to 
%\[ (1-\lambda) + \lambda \alpha \]
%and that moreover that $R_{g_\lambda}>0$ at infinity, for $\lambda>0$ small enough, then for all $\lambda\in(0,\lambda_0)$, we have
%\[ (1-\lambda) + \lambda \alpha > 0.\]
%Then, let $\lambda_1=\sup\{\lambda\; ; \; (1-\mu) + \mu \alpha > 0 \;\; \forall \; \mu\in(0,\lambda) \}$. If $\lambda_1<1$, then $\mathcal{E}_{-1}(g_{\lambda_1})=0$, and by theorem \ref{PEthm3} we must have $Q_{g_{\lambda_1}}=\frac{2}{n-4}u_{\lambda_1}^{-\frac{n+4}{n-4}}(1-\lambda_1)\frac{n-4}{2}Q_g\equiv 0$, which contradicts the fact that $Q_g$ is semi-positive. Hence for all $\lambda(0,1)$, we have,
%\[ (1-\lambda) + \lambda \alpha \geq 0,\]
%hence
%\[ \alpha \geq 0 .\]

%From Propositions \ref{blowup.0} and \ref{blowup-scalcurv2}  we know that $(\hat{M}=M\backslash\{p\},\hat{g}=G_P^{\frac{4}{n-4}}g)$ is an AE manifold whose energy is proportional to $\alpha$ and which satisfies the hypothesis of  theorems \ref{PEthm3} . Therefore if $\alpha=0$, we find that $\hat{M}$ is conformally equivalent to $\mathbb{R}^n$. Being $M$ the one point compactification of $\hat{M}$, we see that $M\cong S^n$ and $g$ is conformal to the round metric.
%\end{proof}

\subsection{The 4-dimensional case}

Finally, we would like to end by briefly discussing some peculiarities concerning the analysis of the four dimensional case. In particular, a large part of what we have done above for the $Q$-curvature positive mass theorem can be translated perfectly well to $n=4$. Nevertheless, in view of Corollary \ref{PEthm4d} it should be clear that strong restrictions should appear in the treatment of this special case. Although the origins of some of these restriction may be evident to experts in $Q$-curvature analysis, the aim in what follows is to make explicit the subtle differences of this case and where the particularities arise. In this spirit, let us start by presenting the following lemma which combines results of \cite{Ndiaye}-\cite{Li}.

\begin{lemma}\label{Greenfunction4d.0}
Let $(M^4,g)$ be a closed Riemannian manifold satisfying $\mathrm{Ker}(P_g)=\mathbb{R}$ and $\kappa_g>0$. Then, given $p\in M$, there exists a Green's function  $G_p\doteq G_{P_g}(p,\cdot)\in C^{\infty}(M\backslash\{p\})$ with a pole at $p$, unique up to an additive constant, which satisfies
\begin{align}\label{Greendunction4d.1}
P_gG_p + Q_g = \kappa_g\delta_p
\end{align}
as distributions. Furthermore, near $p$, in $g$-normal coordinates $\{x^i\}_{i=1}^4$ the following expansion holds
\begin{align}\label{Greendunction4d.2}
G_p&=\frac{\kappa_g}{16\pi^2}\ln(r^{-2}) + \frac{\kappa_g}{16\pi^2}S_0 + \frac{\kappa_g}{16\pi^2}a_ix^i + \frac{\kappa_g}{16\pi^2}b_{ij}x^ix^j + o_4(r^{2}).
\end{align}
for some constants $S_0,a_i,b_{ij}$.
\end{lemma}
\begin{proof}
From Lemma 2.1 in \cite{Ndiaye} we know that if $\mathrm{Ker}(P_g)=\mathbb{R}$, then the Paneitz operator $P_g$ admits a Green function $G_P$, which is to say that for every $u\in C^4(M)$ it holds that
\begin{align}
u(x)-\bar{u}=\int_MG_P(x,y)P_g(u(y))dV_g(y),
\end{align}
where $\bar{u}\doteq \mathrm{vol_g}(M)^{-1}\int_MudV_g$. Also, since $\mathrm{Ker}(P_g)=\mathbb{R}$, we know that there is a smooth function $U$ satisfying
\begin{align*}
P_gU=-(Q_g-\bar{Q}_g).
\end{align*}
Define $G_p\doteq \kappa_g G_P + U$, so that 
\begin{align*}
\langle P_gG_p,u \rangle&=\kappa_g u(p) - \kappa_g \bar{u} + \langle P_gU,u \rangle ,\\
&=\kappa_g u_p - \kappa_g\bar{u} - \langle Q_g,u \rangle + \langle \frac{\kappa_g}{\mathrm{vol}_g(M)},u \rangle =\kappa_g u_p - \langle Q_g,u \rangle. 
\end{align*}
Thus, we see that
\begin{align}
\langle P_gG_p + Q_g,u \rangle= \kappa_g u(p) \;\; \forall u\in C^{\infty}(M),
\end{align}
which is to say that $P_gG_p + Q_g=\kappa_g \delta_p$. The uniqueness claim follows since two different such functions $G_p$ and $\tilde{G}_p$ must satisfy $P_g(G_p-\tilde{G}_p)=0$, implying $G_p=\tilde{G}_p+c$. If we rewrite (\ref{Greendunction4d.1}) as $P_g\left(\frac{16\pi^2}{k_g}G_p\right) + \frac{16\pi^2}{k_g}Q_g=16\pi^2 \delta_p$ we can appeal quite straightforwardly to the computations of the Appendix in \cite{Li} to conclude that, in $g$-normal coordinates, the Green function has the following expansion near $p$:
\begin{align}
\frac{16\pi^2}{\kappa_g}G_p&=-2\ln(r) + S_0 + a_ix^i + b_{ij}x^ix^j + o_4(r^{2}).
\end{align}
which proves (\ref{Greendunction4d.2}).
\end{proof}

Let us now consider the same setting as in the above lemma and consider the inverted coordinates $z^i=\frac{x^i}{r^2}$ and define $(\hat{M}=M\backslash\{p\},\hat{g}=e^{2G_p}g)$. Then, 
\begin{align*}
Q_{\hat{g}}(q)=e^{-4G_p(q)}\left(P_gG_p(q) +  Q_g(q)\right)=0, \text{ for any } q\in\hat{M}.
\end{align*}
and
\begin{align}
\label{Greenfunction4d.1}
\begin{split}
\hat{g}_{ij}(z)&=e^{\ln(\rho^{4\frac{\kappa_g}{16\pi^2}})+ \frac{\kappa_g}{8\pi^2}S_0 +  \frac{\kappa_g}{8\pi^2}a_i\frac{z^i}{\rho^2} + \frac{\kappa_g}{8\pi^2}b_{ij}\frac{z^iz^j}{\rho^4} + O_4(\rho^{-3})}g(\partial_{z^i},\partial_{z^j}),\\
&=\rho^{4\frac{\kappa_g}{16\pi^2}}e^{\frac{\kappa_g}{8\pi^2}S_0}e^{\frac{\kappa_g}{8\pi^2}a_i\frac{z^i}{\rho^2} + \frac{\kappa_g}{8\pi^2}b_{ij}\frac{z^iz^j}{\rho^4} + O_4(\rho^{-3})}\rho^{-4}\left(\delta^i_k - 2\rho^{-2}z^iz^k \right)\left(\delta^j_l - 2\rho^{-2}z^jz^l \right)g_{kl}(\rho^{-2}z),\\
&=\rho^{4\Delta\kappa_g}e^{\frac{\kappa_g}{8\pi^2}S_0}e^{\frac{\kappa_g}{8\pi^2}a_i\frac{z^i}{\rho^2} + \frac{\kappa_g}{8\pi^2}b_{ij}\frac{z^iz^j}{\rho^4} + O_4(\rho^{-3})}\left(\delta_{ij} + O_4(\rho^{-2}) \right),\\
&=\rho^{4\Delta\kappa_g}e^{\frac{\kappa_g}{8\pi^2}S_0}\left(1 + \frac{\kappa_g}{8\pi^2}a_i\frac{z^i}{\rho^2} +  O_4(\rho^{-2}) \right)\left(\delta_{ij} + O_4(\rho^{-2}) \right),
\end{split}
\end{align}

where we have defined $\Delta\kappa_g\doteq \frac{\kappa_g}{16\pi^2}-1$. Let us notice that the above inversion gives an AE gemtric $\hat{g}$ iff $\Delta\kappa_g=0$, which is to say $\kappa_g=16\pi^2$. In particular, in this case, via a coordinate change of the form $\bar{z}^i=e^{S_0}z^i$, we find that
\begin{align}
\hat{g}_{ij}(\bar{z})&=\left(1+2e^{S_0}a_k\frac{\bar{z}^k}{\bar{\rho}^2} + O_4(\bar{\rho}^{-2})  \right)\left(\delta_{ij} + O_4(\bar{\rho}^{-2}) \right),
\end{align}
where $\bar{\rho}=|\bar{z}|=e^{S_0}\rho$. That is, $\rho^{-k}=e^{kS_0}\bar{\rho}^{-k}$. We see that $(\hat{M},\hat{g})$ is an AE-manifold or order $\tau=1$. Therefore, since by construction $Q_{\hat{g}}\equiv 0$, if $Y(g)> 0$, in view of Corollary \ref{PEthm4d}, we must conclude that $(M,g)$ is conformal to the round sphere. That is, the above computations together with Corollary \ref{PEthm4d} imply the following Corollary, which is not new, since it concerns the rigidity statement involved in Theorem B in \cite{Gursky}.
\begin{coro}
Let $(M,g)$ be a closed 4-dimensional manifold which satisfies $\mathrm{Ker}(P_g)=\mathbb{R}$ and $\kappa_g=16\pi^2$. If $Y(g)>0$, then $(M,g)$ is conformal to the round sphere.
\end{coro}

%\bigskip
%Let us highlight that the computations preceding the above corollary actually show what would go wrong if we intended to prove Proposition \ref{blowup-scalcurv} in dimension four. In fact, the only thing that fails is that when applying the continuity argument which establishes $R_{g_{\lambda}}\geq 0$, in the appeal to the maximum principle, we need to guarantee that all these metrics are AE. Just as in the above computations we needed to restrict to $\kappa_g=16\pi^2$ to obtain this asymptotic behaviour, it is easy to see that the corresponding metrics $g_{\lambda}=e^{2\lambda G_p}g$ are only AE for $\lambda=1$ (and $\kappa_g=16\pi^2$), which would make the continuity argument flawed.

%Nevertheless, if $\kappa_g \not= 16\pi^2$, the blow-up manifold is asymptotic to some flat cone and for this reason the mass can tsill be computed and we get the following corrolary which is an improvement of 
%lemma page 134 of \cite{Gursky}, see also Proposition E \cite{Gursky}, corollary 5.1 \cite{hangyang2}. 
Finally, let us highlight that the techniques developed so far allow us to actually get more than the previous corollary. In fact, we can recover the full statement of Theorem B in \cite{Gursky} by an independent and simple proof. 

\begin{thm}[Gursky]
Let $(M^4,g)$ a $4$-dimensional manifold with $Y([g])\geq 0$, then $\kappa_g\leq 16\pi^2$ with equality holding iff $(M^4,g)$ is conformal to the standard sphere.
\end{thm}
\begin{proof}
Since we need only pay attention to the cases $\kappa_g>0$, let us start by noticing that, due to Theorem A in \cite{Gursky}, under our hypotheses $\mathrm{Ker}(P_g)=\mathbb{R}$. Furthermore, the case $Y([g])=0$ is also trivial, since in this case $\kappa_g\leq 0$, thus we will assume $Y([g])>0$. Since our hypotheses are conformally invariant, let us start assuming that $g$ has been picked within $[g]$ so as to be the metric of a conformal normal coordinate system. We can now appeal to Lemma \ref{Greenfunction4d.0} and the expansion (\ref{Greenfunction4d.1}) to construct the manifold $(\hat{M}=M\backslash\{p\},\hat{g})$. Alhtough in general this manifold will not be AE, it holds that $Q_{\hat{g}}\equiv 0$. Also, since $Y([g])>0$, we now that the Green function $G_{L}$ of the conformal Laplacian (with a pole at $p\in M$) exists, it is positive and therefore we can construct the metric $\tilde{g}\doteq G^{2}_{L}g$ on $\hat{M}$. In inverted normal coordinates around $p$, from \cite{Lee-Parker}, we know that 
\begin{align*}
\tilde{g}(z)(\partial_{z^i},\partial_{z^j})=\delta_{ij}+O_4(\rho^{-2}),
\end{align*}  
which shows that $(\hat{M},\tilde{g})$ is AE of order $\tau=2$. Furthermore, $\hat{g}$ and $\tilde{g}$ are related via
\begin{align*}
\hat{g}=e^{-\ln(G^2_L)}e^{2G_p}\tilde{g}=e^{2(G_p-\ln(G_L))}\tilde{g}
\end{align*}
and therefore, from the conformal covariance associate to the Paneitz operator, we see that
\begin{align}
0=e^{4\Phi}Q_{\hat{g}}=P_{\tilde{g}}\Phi + Q_{\tilde{g}},
\end{align}
where we have defined $\Phi\doteq G_{P} - \ln(G_L)$. Following the proof theorem \ref{PEthm}, we know that $Q_{\tilde{g}}=-\frac{1}{2}|\mathrm{Ric}_{\tilde{g}}|^2_{\tilde{g}}$, and therefore we find that, for $\rho$ large enough,
\begin{align}
\int_{D_{\rho}}\frac{1}{2}|\mathrm{Ric}_{\tilde{g}}|^2_{\tilde{g}}dV_{\tilde{g}}=\int_{S_{\rho}}\tilde{g}(\tilde{\nabla}\Delta_{\tilde{g}}\Phi,\tilde{\nu})d\tilde{\omega} + 2\int_{S_{\rho}}\mathrm{Ric}_{\tilde{g}}(\tilde{\nabla}\Phi,\tilde{\nu})d\tilde{\omega}.
\end{align}
Now, the main difference with respect to theorem \ref{PEthm} is that we cannot estimate the derivatives of $\Phi$ in the same way, since it does not solve the same equation as in that proof. Nevertheless, in this case, we have an explicit expression for $\Phi$, at least asymptotically. Thus, let us notice that for sufficiently large $\rho$ the following holds
\begin{align*}
\Phi&=\ln(\rho^{2\alpha}) + S_0 + O_{4}(\rho^{-1}) - \ln(\rho^{2}+A+O_4(\rho^{-1})),\\
%&=\ln(\rho^{2\alpha}) + S_0 + O_{4}(\rho^{-1}) - \ln(\rho^{2}(1+A\rho^{-2}+O_(4)(\rho^{-3})),\\
&=\ln(\rho^{2\alpha}) + S_0  - \ln(\rho^{2}) - \ln(1+A\rho^{-2}+O_4(\rho^{-3})) + O_{4}(\rho^{-1}),\\
&=\ln(\rho^{2(\alpha-1)}) + S_0 - \ln(1+A\rho^{-2}+O_4(\rho^{-3})) + O_4(\rho^{-1}),
\end{align*}
where above we have defined $\alpha\doteq \frac{\kappa_g}{16\pi^2}$ and both $A$ and $S_0$ are constants. From all this we can directly compute that
\begin{align*}
\tilde{\nabla}_i\Phi&=2(\alpha-1)\frac{z^{i}}{\rho^2}+O_3(\rho^{-2}),\\
%\tilde{\nabla}_j\tilde{\nabla}_i\Phi&=2(\alpha-1)\left(\frac{\delta^{i}_j}{\rho^2}-2\frac{z^iz^j}{\rho^{4}} + \underbrace{\tilde{\Gamma}^{k}_{ij}}_{O_3(\rho^{-3})}\underbrace{\tilde{\nabla}_k\Phi}_{O_3(\rho^{-1})}\right),\\
\Delta_{\tilde{g}}\Phi&=\frac{4(\alpha-1)}{\rho^2} + O_2(\rho^{-3}),\\
\tilde{\nabla}_i\Delta_{\tilde{g}}\Phi&=-\frac{8(\alpha-1)}{\rho^3}\frac{z^i}{\rho} + O_2(\rho^{-4}).
\end{align*}
From the above expression, it follows that
\begin{align*}
\mathrm{Ric}_{\tilde{g}}(\tilde{\nabla}\Phi,\tilde{\nu})&=o(\rho^{-3}),\\
\tilde{g}(\tilde{\nabla}\Delta_{\tilde{g}}\Phi,\tilde{\nu})&=\frac{8(1-\alpha)}{\rho^3} + O_1(\rho^{-4}),
\end{align*}
implying that
\begin{align}
\int_{D_{\rho}}\frac{1}{2}|\mathrm{Ric}_{\tilde{g}}|^2_{\tilde{g}}dV_{\tilde{g}}=8\omega_3(1-\alpha)+o(1).
\end{align}
Finally, passing to the limit as $r$ goes to infinity, we find that
\begin{align}
\int_{\hat{M}}|\mathrm{Ric}_{\tilde{g}}|^2_{\tilde{g}}dV_{\tilde{g}}=8\omega_3(1-\alpha)\geq 0.
\end{align}
This implies that $\alpha\leq 1$ and the equality case has already been dealt with in the previous corollary, which establishes the theorem.

\end{proof}

\subsection{The 3-dimensional case}
In this section, we briefly explain how  to recover the following theorem due to Hang and Yang \cite{hangyang3}.

\begin{thm}[Proposition 2.4 \cite{hangyang3}]
Assume the Yamabe invariant $Y([g]) > 0$, $\ker P_g = 0$. If there is some $p\in M$ such that $G_{p}\doteq G_{P_g}(p,\cdot) < 0$ on $M\backslash\{p\}$, then $G_{p}(p)< 0$ except when $(M,g)$ is conformally equivalent to the standard $S^3$.
\end{thm}

%\todo[author=Rodrigo]{In this case, the proof is actually almost the same as in \cite{hangyang3}. The only difference is the computations are made in inverted coordinates. I don't know if it's enough to keep this part or maybe should just comment on their result.}

In this case the proof runs along the same lines as in the previous theorem. In particular, its hypotheses and conclusions are conformally invariant. Thus, we can assume that $g$ is the metric of a conformal normal coordinate system. Then, let $\hat{M}=M\backslash\{p\}$ and define $\hat{g}\doteq G^{-4}_{p}g$ and $\tilde{g}=G^4_{L}g$ on this manifold, where, as in the proof of the previous theorem, $G_L$ denotes the Green function of the conformal Laplacian with a pole at $p$. Let us recall the following expansions, valid in conformal normal coordinates around $p$ (see, \cite{hangyang3} and \cite{Lee-Parker})
\begin{align*}
G_p&=A+O_4(r),\\
G_L&=\frac{1}{r}+\alpha+O_4(r),
\end{align*}
where $A$ and $\alpha$ are constants. After going to inverted coordinates $z=\frac{x}{r}$, with $\rho=r^{-1}$, we   find that
\begin{align*}
\tilde{g}_{ij}(z)=\left(1+\frac{\alpha}{\rho}\right)\delta_{ij} +O(\rho^{-2}),
\end{align*}
and, clearly, on $\hat{g}=\Phi^{-4}\tilde{g}$, with $\Phi\doteq G_PG_L$ and $Q_{\hat{g}}\equiv 0$. Therefore, from the conformal covariance of the Paneitz operator, we find that $P_{\tilde{g}}\Phi\equiv 0$, which translates to
\begin{align*}
0=\Delta^2_{\tilde{g}}\Phi + 4\mathrm{div}_{\tilde{g}}\left(\mathrm{Ric}_{\tilde{g}}(\nabla \Phi,\cdot) \right)+|\mathrm{Ric}_{\tilde{g}}|^2_{\tilde{g}}\Phi.
\end{align*} 
That is
\begin{align}
\int_{D_{\rho}}|\mathrm{Ric}_{\tilde{g}}|^2_{\tilde{g}}\Phi dV_{\tilde{g}}=-\int_{S_{\rho}}\tilde{g}(\tilde{\nabla}\Delta_{\tilde{g}}\Phi,\tilde{\nu})d\tilde{\omega}_{\rho} - 4\int_{S_{\rho}}\mathrm{Ric}_{\tilde{g}}(\nabla\Phi,\tilde{\nu})d\tilde{\omega}_{\rho}.
\end{align}
As in the previous theorem, we now can compute explicitly the terms in the right-hand side. That is, 
\begin{align*}
\Phi&=\left(A+O_4(\rho^{-1}) \right)\left(\rho + \alpha +O_4(\rho^{-1}) \right)=A\rho + O_4(\rho^{0}) \\
\nabla_i\Phi&=A\frac{z^i}{\rho} + O_3(\rho^{-1}),\\
%\tilde{\nabla}_{j}\tilde{\nabla}_{i}\Phi&=A\left(\frac{\delta_{ij}}{\rho} - \frac{z^iz^j}{\rho^3} \right) + O_2(\rho^{-2}),\\
\Delta_{\tilde{g}}\Phi&=\frac{2A}{\rho} + O_2(\rho^{-2}),\\
\nabla_i\Delta_{\tilde{g}}\Phi&=-2A\frac{z^i}{\rho^{3}} + O_1(\rho^{-3}),
\end{align*}
which, together with $\mathrm{Ric}_{\tilde{g}}=O_2(\rho^{-3})$ implies that
\begin{align*}
\int_{D_{\rho}}|\mathrm{Ric}_{\tilde{g}}|^2_{\tilde{g}}\Phi dV_{\tilde{g}}=8\pi A + O(\rho^{-1}).
\end{align*}
Passing to the limit and remembering that $\Phi<0$, we find that $A\leq 0$ with equality holding iff $\mathrm{Ric}_{\tilde{g}}\equiv 0$. That is, if $\hat{M}\cong \mathbb{R}^3$ and $\tilde{g}=\delta$, which implies the final result.

%Indeed the metric $\hat{g}=(-G_p)^{-4} g$ has vanishing $Q$-courbure, then as the proof of theorem \ref{PEthm}, we obtain that, with $\tilde{g}=(G_L)^4g=\Phi^{-4}\hat{g}$, we find that 
%\begin{align}- \int_{M} |\mathrm{Ric}_{\tilde{g}}|^2_{\tilde{g}}\Phi \, dv_{\tilde{g}} = -\lim_{r\rightarrow 0} \int_{S(p,r)}\partial_r R_{g} d\omega_g =G_p(p),
%\end{align}
%which proves the theorem.

\appendix
\markboth{Appendix}{Appendix}
\renewcommand{\thesection}{\Alph{section}}
\numberwithin{equation}{section}
\numberwithin{thm}{section}
\numberwithin{coro}{section}
\numberwithin{remark}{section}

\section{Appendix: Some analytic results concerning AE manifolds}

In this appendix we will collect some facts results concerning AE manifolds which are used in the core of the paper. Most of these results are well-known for experts. We include them for the sake completeness and to deliver a self-contained presentation. For detailed proofs and discussions on these topics, we refer the reader to Bartnik \cite{Bartnik}, Lee-Parker \cite{Lee-Parker}.

Let us start with the following fundamental theorem regarding the properties of the Laplacian on AE manifolds.
    \begin{thm}
        \label{deltag}
        Let $(M,g)$ an asymptotically Euclidean manifold with a structure at infinity $\phi :M\setminus K \rightarrow \R^n\setminus B_1$  with decay rate $\tau$. If $\delta \not\in (\mathbb{Z} \setminus \{-1,\cdots,3-n\})$ and $q>1$ then
        $$\Delta_g: W_\delta^{2,q}(\phi) \rightarrow L^q_{\delta-2}$$
        is Fredholm. Moreover 
        \begin{equation}
        \left\{
        \begin{array}{l}
        \hbox{ if } \delta >2-n \hbox{ then } \Delta_g \hbox{ is surjective,}\\
        \hbox{ if } 2-n<\delta <0 \hbox{ then } \Delta_g \hbox{ is bijective,}\\
        \hbox{ if } \delta <0 \hbox{ then } \Delta_g \hbox{ is injective.}
        \end{array}
        \right.
        \end{equation}
    \end{thm}

\begin{remark}
The decay rates in the set $\mathbb{Z} \setminus \{-1,\cdots,3-n\}$ are called \textbf{exceptional} and we say that $\delta$ is \textbf{non-exceptional} if $\delta\not\in \mathbb{Z} \setminus \{-1,\cdots,3-n\}$. 
\end{remark}
Let us now analyse the relation between different potential structures of infinity. In particular, the following theorem concerns the existence of harmonic coordinates.

\begin{thm}
Let $(M,g)$ be an AE manifold, with $g\in W^{k,q}_{loc}$, $k\geq 1$, $q>n$, and $(\Phi,x):M\backslash K\mapsto E_R\doteq \mathbb{R}^n\backslash\overline{B_R(0)}$ where $K\subset\subset M$, $R\geq 1$ is a structure of infinity of order $\tau> 0$ with $1-\tau$ non-exceptional, and fix $1<\eta<2$. There are functions $y^{i}\in W^{k,q}_{\eta}$, $i=1,\cdots,n$, such that $\Delta_gy^i=0$ and $(x^i-y^i)\in W^{k,q}_{1-\tau^{*}}(E_R)$ for $\tau^{*}\doteq \min\{\tau,n-2\}$, which implies
\begin{align}\label{harmoniccoord1}
\begin{split}
|x^i-y^i|&=o_k(r^{1-\tau^{*}}),\\
|g(\partial_{x^i},\partial_{x^j})-g(\partial_{y^i},\partial_{y^j})|&=o_k(r^{-\tau^{*}}).
\end{split}
\end{align}
Furthermore, the set of functions $\{1,y^i\}$ is a basis for $H_1=\{u\in W^{k,q}_{\eta}: \Delta_gu=0\}$. 
\end{thm}
\begin{proof}
Let us first extend the functions $x^i$ smoothly to all of $M$. Then, near infinity, $\Delta_gx^i=g^{kl}\Gamma^i_{kl}\doteq \Gamma^i\in W^{k-1,q}_{-1-\tau}$. Notice that if $1-\tau>2-n$, then Theorem \ref{deltag} implies the existence of $v^i\in W^{k+1,q}_{1-\tau}$ solving $\Delta_gv^i=\Gamma^i$. Now, if $1-\tau< 2-n$ (the equality case is an exceptional case), we cannot, a priori, improve the decay given by $r^{2-n}$. This actually follows from Theorem 1.17 in \cite{Bartnik}. Therefore, in any case, we know that there are solutions $v^i\in W^{k+1,q}_{1-\tau^{*}}$ to $\Delta_gv^i=\Gamma^i$, where $\tau^{*}=\min\{\tau,n-2\}$. Therefore, there is some $v^i\in W^{k+1,q}_{1-\tau^{*}}$ satisfying $\Delta_g(x^i-v^i)=0$. Let us then define $y^i\doteq x^i-v^i$, which implies the first estimate in (\ref{harmoniccoord1}), and furthermore, since $\frac{\partial y^i}{\partial x^j}=\delta^i_j + o(r^{-\tau^{*}})$, we see that near infinity $\{y^i\}$ are coordinates which are asymptotically Cartesian. This, in turn, implies the second estimate in (\ref{harmoniccoord1}). The final claim concerning $H_1$ follows since $\{1,y^i\}$ spans an $(n+1)$-dimensional subspace of $H_1$, but from Proposition 2.2 in \cite{Bartnik} $\mathrm{dim}(H_1)=n+1$. 
\end{proof}

\begin{remark}
Let us highlight that the statement of the above theorem is slightly different than the well-known Theorem 3.1 in \cite{Bartnik}. The difference relies in the fact that, in \cite{Bartnik}, it does not seem to be explicitly stated that, a priori, the decay rate $r^{2-n}$ cannot be improved, regardless of how large $\tau$ may be. Nevertheless, as can be seen in the above proof, the above mild correction comes about from results contained in the same reference, as is also clear from a careful examination of the proof of Theorem 3.1 in \cite{Bartnik}. Similar comments apply to the following theorem,  which can be found in \cite{Bartnik} as Corollary 3.2, and presents the relation between two different structures of infinity.
\end{remark}

%The following theorem, which can be found in \cite{Bartnik} as Corollary 3.2, presents the relation between two different structures of infinity.
\begin{thm}\label{Asymptoticcoord}
Let $(M,g)$ an asymptotically flat manifold, $g\in W^{k,q}_{loc}$, $k\geq 1$ and $q>n$, with two structures at infinity $\phi,\psi :M\setminus K \rightarrow \R^n\setminus \overline{B_1(0)}$  with decay rates $\tau_\phi$ and $\tau_\psi$, where of each of these weights satisfies that $1-\tau$ is non-exceptional. There exists $(O,a)\in O(n)\times\R^n$ such that
        $$x^i-(O^i_jz^j +a^i) \in W^{k,q}_{1-\tau}(\R^n)$$
    which implies
		\begin{align}
		\vert x^i-(O^i_jz^j +a^i)\vert=o_{k}(r^{1-\tau}),
		\end{align}
where $\tau\doteq \min\{\tau_\phi,\tau_\psi,n-2\}$, $x=\phi^{-1}$ and $z=\psi^{-1}$.
\end{thm}
\begin{proof}
Let $y^i$ and $w^i$ be the harmonic coordinates constructed in the previous theorem associated to $\phi$ and $\psi$ respectively. Then, since $H_1$ is intrinsic to $M$ and $\{1,y^i\}$ and $\{1,w^i\}$ are bases for this space, we get that
\begin{align*}
w^i=A^i_jy^j + a^i,
\end{align*}
where we $A\in \mathrm{GL}(n,\mathbb{R})$ a priori, but actually, from the construction of the previous theorem, we know that $z^i=A^i_jx^j + a^i - A^i_jv^j + \bar{v}^i$, with $v^i\in W^{k,q}_{1-\tau^{*}_{\phi}}$ and $\bar{v}^i\in W^{k,q}_{1-\tau^{*}_{\psi}}$, where $\tau^{*}$ was defined in the previous theorem. Since these last two systems are Cartesian, then $A\in O(n)$ and we explicitly see that $z^i-A^i_jx^j-a^i\in W^{k,q}_{1-\tau}$.
\end{proof}

    The above estimate is the best possible and it is determined by the Ricci curvature as shown by the next theorem which corresponds to Proposition 3.3 in \cite{Bartnik}, see also \cite{DK}.
\begin{thm}
Let $(M,g)$ an asymptotically flat manifold with a structure at infinity $\phi :M\setminus K \rightarrow \R^n\setminus B(0,1)$  with decay rate $\tau$ such that $(\phi_{*} g-\delta)\in W^{2,q}_{-\tau}(\R^n\setminus B(0,1))$ for $q>n$ and such that 
$$Ric_g \in L^q_{-2-\eta}(M) \hbox{ for some } \eta >\tau \hbox{ and } \eta \not\in (\mathbb{Z} \setminus \{-1,\cdots,3-n\}) .$$
Then there exists a structure at infinity $\Theta  :M\setminus K' \rightarrow \R^n\setminus B_1$ such that $(\Theta_* g-\delta)\in W^{2,q}_{-\eta}(\R^n\setminus B_1)$.
\end{thm}

%Hence in our setting where we assume that there exists a structure such that $g_{ij}=\delta_{ij} +O_4(r^{-\tau})$, every statement will be independent of the structure at infinity as soon as the weight of the Sobolev space is greater than $\tau$\todo[author=paul, color=blue!50!white]{It is the contrary no?}. 

Finally, the following statement concerning the conformal Laplacian will be important in our analysis. It is a variant of theorem 9.2 of \cite{Lee-Parker}

\begin{coro}
    \label{CLg}
    Let $(M,g)$ an asymptotically flat manifold of order $\tau>0$ and assume $R_g\geq 0$. If $2-n<\delta<0$ and $q>1$, then     $$L_g =\Delta_g -c_n R_g : W_\delta^{2,q} \rightarrow L^q_{\delta-2}$$
    is an isomorphism.
\end{coro}

\begin{proof} Since $L_g$ is a compact perturbation of the Laplacian, it is also Fredholm, hence thanks to its self-adjointness it suffices to proof injectivity to get the result.
Consider $ u\in W_\delta^{2,q}$ such that $\Delta_g u=c_n R_g u$, and first notice that by elliptic regularity we can assume that $u\in W^{2,q}_{\delta}$ with $q>n$. Also, notice that $R_g u \in W^{2,q}_{\delta'-2}$ for any $\delta'>\delta-\tau$. Since $\tau>0$, we can pick $\delta'$ satisfying $\max\{\delta-\tau,2-n\}<\delta'<\delta$ and appeal to Theorem \ref{deltag} to obtain $u\in W^{2,q}_{\delta'}$. Since we can repeat the argument as much as necessary, we can assume that $u\in W^{2,q}_{\delta'}$ with $\delta'> 2-n$ taken arbitrary close to $2-n$. Then, we can multiply $L_g(u)=0$ by $u$ and integrate by parts to obtain
$$\int_{M\setminus(\R^n \setminus B_R(0))} \left(\vert \nabla u \vert_g^2 + c_n R_g u^2\right) dv_g  = \int_{\partial B_R(0)}  u\partial_\nu u \; d\sigma,$$
Since $u=O_2(r^{\delta'})$ near infinity, we can estimate $\vert \nabla u \vert_g^2=O(r^{2\delta'-2})$ and $R_gu^2=O(r^{2\delta'-\tau-2})$ which implies these quantities are in $L^1(M)$ as long as chose $2-n<\delta'<\frac{2-n}{2}$. Furthermore, under these condition one finds $u\partial_\nu u=O(r^{2\delta'-1})$ and therefore 
\begin{align*}
\int_{M\setminus(\R^n \setminus B_R(0))} \left(\vert \nabla u \vert_g^2 + c_n R_g u^2\right) dv_g  = O(R^{n-2\delta'-2})
\end{align*}
so that we can pass to the limit as $R\rightarrow +\infty$, which gives that $u\equiv 0$ since $R_g\textcolor{blue}{\geq}0$ and $u\rightarrow 0$ as $|x|\rightarrow\infty$, which concludes the proof.
\end{proof}

%\bigskip
%Of course we can define asymptotically flat structures with multiple ends. But since analysis phenomena are determined by the behaviour at infinity, it is very easy to isolate each ends and to consider that there is only one.

\section{Appendix: Conventions on Q-curvature}

Let us adopt the following general definition of $Q$-curvature for an arbitrary Riemannian manifold $(M^n,g)$ with $n\geq 3$:
\begin{align}\label{Qcurv}
Q_g=-\Delta_g\sigma_1(S_g) + 4\sigma_2(S_g) + \frac{n-4}{2}\left(\sigma_1(S_g)\right)^2, 
\end{align} 
where $S_g\doteq \frac{1}{n-2}\left(\mathrm{Ric}_g - \frac{1}{2(n-1)}R_gg \right)$ stands for the Schouten tensor and $\sigma_k(S_g)$ stands for the $k$-th elementary symmetric function of the eigenvalues of $S_g$. In this context, the Paneitz operator is defined by
\begin{align}\label{Paneitz}
P_gu\doteq \Delta^2_gu + \mathrm{div}_{g}\left(\left( 4S_g - (n-2)\sigma_1(S_g)g \right)(\nabla u,\cdot)\right) + \frac{n-4}{2}Q_gu,
\end{align}
for all $u\in C^{\infty}(M)$. In this context, the following relations hold:
\begin{align}
\begin{split}
\sigma_1(S_g)&=\frac{R_g}{2(n-1)},\\
\sigma_2(S_g)&=\frac{1}{2}\left( \frac{n^2-n}{4(n-1)^2(n-2)^2}R^2_g - \frac{|\mathrm{Ric}_g|^2_{g}}{(n-2)^2} \right),
\end{split}
\end{align}
which implies that
\begin{align}
\begin{split}
Q_g&=-\frac{1}{2(n-1)}\Delta_gR_g - \frac{2}{(n-2)^2}|\mathrm{Ric}_g|^2_g + \frac{n^3-4n^2+16n-16}{8(n-1)^2(n-2)^2}R_g^2,\\
P_gu&=\Delta^2u + \mathrm{div}_g\left(\left(\frac{4}{n-2}\mathrm{Ric}_g - \frac{n^2-4n+8}{2(n-1)(n-2)}R_g\: g \right)(\nabla u,\cdot) \right) + \frac{n-4}{2}Q_gu.
\end{split}
\end{align}
In particular, for $n\neq 4$, if $\bar{g}=u^{\frac{4}{n-4}}g$, then
\begin{align}\label{Qcurvtranf}
Q_{\bar{g}}=\frac{2}{n-4}u^{-\frac{n+4}{n-4}}P_gu.
\end{align}
In the case of $n=4$ we can apply the above definitions to get
\begin{align}
\begin{split}
Q_g&=-\frac{1}{6}\Delta_gR_g - \frac{1}{2}|\mathrm{Ric}_g|^2_g + \frac{1}{6}R_g^2,\\
P_gu&=\Delta^2_gu + \mathrm{div}_g\left(\left( 2\mathrm{Ric}_g - \frac{2}{3}R_g\:g \right)(\nabla u,\cdot)\right),
\end{split}
\end{align}
and in this case, if $\bar{g}=e^{2u}g$, we have that \cite{Lin-Yuan}
\begin{align}\label{Qcurvtrans4d.1}
Q_{\bar{g}}=e^{-4u}\left(P_gu+Q_g \right).
\end{align}

Now, let us notice that it is also quite standard to \textit{redefine} the $Q$-curvature in dimension 4 via (see \cite{Malchiodi1,Li})
\begin{align}
Q^{(4)}_g=\frac{1}{2}Q_g=-\frac{1}{12}\Delta_gR_g - \frac{1}{4}|\mathrm{Ric}_g|^2_g + \frac{1}{12}R_g^2
\end{align}
and in this case, it follows that
\begin{align}
2Q^{(4)}_{\bar{g}}=e^{-4u}\left(P_gu+2Q^{(4)}_g \right).
\end{align}
We will not adopt this redefinitions and keep a unified notation via (\ref{Qcurv}) along this paper.

\section{Appendix: Conformal normal coordinates}

In order to deliver a presentation as self-cointained as possible, this section is meant to summarise some of the results concerning conformal normal coordinates presented in \cite{Lee-Parker}, which which are used in the core of this paper. The basic construction is given on a smooth Riemannian manifold $(M^n,g)$ where we intend to expand $g$ around a fixed point $p\in M$. In particular, we are interested in finding an element $\tilde{g}$ within the conformal class $[g]$ for which, in $\tilde{g}$-normal coordinates $\{x^i\}$, the following expansion holds around $p$:
\begin{align}\label{confcoord.0}
\det(\tilde{g})=1+O(r^{N})
\end{align}
for any chosen $N\geq 2$, where $r=|x|$ (see Theorem 5.1 in \cite{Lee-Parker}). This type of expansion is achieved by first noticing that, for any Riemannian metric $g$, in $g$-normal coordinates $\{x^i \}_{i=1}^n$ around $p\in M$, the following holds
\begin{align*}
g_{ij}(x)=\delta_{ij}+\frac{1}{3}R_{iklj}x^kx^l + \frac{1}{6}R_{iklj,a}x^kx^lx^a+\left(\frac{1}{20}R_{iklj,ab} + \frac{2}{45}R_{iklc}R_{jabc} \right)x^kx^lx^ax^b + O(r^{5}),
\end{align*}
where all the coefficients at evaluated at $p$. From this expression it is possible to compute $\det(g)$ around $p$ for any such metric, so as to get
\begin{align}\label{det-exp}
\begin{split}
\det(g)(x)&=1-\frac{1}{3}R_{ij}x^ix^j - \frac{1}{6}R_{ij,k}x^ix^jx^k \\
&- \left(\frac{1}{20}R_{ij,kl} + \frac{1}{90}R_{aijb}R_{aklb} -\frac{1}{18}R_{ij}R_{kl} \right)x^ix^jx^kx^l + O(r^5),
\end{split}
\end{align}
where again all the coefficients are evaluated at $p$.  Assuming an expanssion of the form $\det(g)=1+O(r^N)$ with $N\geq 2$ (the case $N=2$ is valid for any metric in its own normal coordinates) and appelaing to Theorem 5.2 in \cite{Lee-Parker}, the authors can find a homogeneous polynomial $f\in\mathcal{P}_N$ so that the expansion of $\det(\tilde{g})=1+O(r^{N+1})$ for $\tilde{g}=e^{2f}g$, which establishes an inductive proof (see the proof of Theorem 5.1). But notice that this implies that the symmetrization of the coefficients in (\ref{det-exp}) of order up to $N$ must vanish for $\tilde{g}$. Thus, in the case we do this construction for $N\geq 4$, we see that this implies
\begin{align}\label{confcoord.1}
\begin{split}
&\tilde{R}_{ij}(p)=0,\\
&\tilde{R}_{ij,k} + \tilde{R}_{ki,j} + \tilde{R}_{jk,i}(p)=0.
\end{split}
\end{align}
Putting together the second condition above with the contracted Bianchi identities, we also see that
\begin{align}\label{confcoord.2}
\tilde{R}_{,k}(p)=0.
\end{align}

%\section*{Curvature conventions}

%Duringt the core of the text, I (Rodrigo) have been following the conventions adopted by Choquet-Bruhat \cite{CB1}. In particular, I define the curvature tensor as
%\begin{align}\label{curvature1}
%R(X,Y)Z=\nabla_X\nabla_YZ - \nabla_Y\nabla_XZ - \nabla_{[X,Y]}Z.
%\end{align}
%This convention is the opposite to O'Neill, that is, $R_{ON}=-R$. Nevertheless, we also der on how we label the components of this tensor. That is, in my notations, I define
%\begin{align}
%\begin{split}
%R^{i}_{jkl}&=dx^{i}(R(\partial_l,\partial_k)\partial_j),\\
%&=\partial_l\Gamma^{i}_{kj} - \partial_k\Gamma^{i}_{lj} + \Gamma^{i}_{lu}\Gamma^u_{jk} - \Gamma^{i}_{ku}\Gamma^u_{jl},\\
%&={R_{{}_{ON}}}^i_{jkl}.
%\end{split}
%\end{align}
%That is, when it comes to the components of the curvature tensor, we have the same convention, despite the fact that when we talk about the operator we use opposite conventions. This comes down to whether we relate the $\partial_l$ and $\partial_k$ to the third or fourth index of $R$. Then, concerning the Ricci tensor, we have that
%\begin{align}
%\mathrm{Ric}_{ij}\doteq R^l_{ijl} = \mathrm{Ric}_{{ON}_{ij}}.
%\end{align}
%\todo[author=paul]{ It is not usefull to change the biblio file with different version, we should keep only one . bib}

\bibliographystyle{plain}
\bibliography{biblio}

\end{document}